\documentclass[11pt]{article}

\usepackage{epsfig}
\usepackage{amssymb, latexsym}
\usepackage{amscd}
\usepackage[all,tips,2cell]{xy}
\usepackage{epsf}
\usepackage[mathscr]{eucal}

\usepackage[titletoc]{appendix}
\usepackage[sc]{titlesec}

\usepackage{amsmath, amssymb, amsthm, amsfonts, mathrsfs, eucal, enumerate, layout}
\usepackage[normalem]{ulem}
\usepackage[all,tips,2cell]{xy}
\usepackage[usenames]{color}
\usepackage{xspace}
\usepackage{ifpdf}
\usepackage{ifthen}
\usepackage{hyperref}
\usepackage{tikz}
\usepackage{tikz-cd}
\usepackage{verbatim}
\usetikzlibrary{matrix,arrows}
\usepackage{mathtools}
\usepackage[usenames]{color}
\usepackage{xcolor}

\long\def\comment#1\endcomment{}

\theoremstyle{plain}

\newtheorem{theorem}{\sc Theorem}[section]
\newtheorem{lemma}[theorem]{\sc Lemma}

\newtheorem{prop}[theorem]{\sc Proposition}
\newtheorem{coroll}[theorem]{\sc Corollary}

\theoremstyle{plain}

\newtheorem{defn}[theorem]{\sc Definition}

\theoremstyle{exercise}
\newtheorem{remark}[theorem]{\sc Remark}

\newtheorem{example}[theorem]{\sc Example}

\makeatletter \@addtoreset{equation}{section} \makeatother

\def\eqref#1{\thetag{\ref{#1}}}

\let\latexref=\ref
\def\ref#1{{\normalfont{\latexref{#1}}}}

\newcommand{\ldot}{{\:\raisebox{2,3pt}{\text{\circle*{1.5}}}}}
\newcommand{\udot}{{\:\raisebox{3pt}{\text{\circle*{1.5}}}}}
\def\dlim_#1{{\displaystyle\lim_{#1}}^\hdot}

\newcommand{\C}{{\mathcal C}}

\newcommand{\D}{{\mathcal D}}

\newcommand{\id}{\operatorname{\rm id}}

\newcommand{\Ob}{\mathrm{Ob}}

\renewcommand{\Bar}{\mathrm{Bar}}
\newcommand{\Cobar}{\mathrm{Cobar}}
\newcommand{\FCobar}{\mathrm{FCobar}}

\newcommand{\Hom}{\mathrm{Hom}}

\newcommand{\RHom}{\mathrm{RHom}}
\newcommand{\Hoch}{\mathrm{Hoch}}

\newcommand{\dg}{{dg}}
\newcommand{\Ho}{\mathrm{Ho}}
\newcommand{\Cone}{\mathrm{Cone}}

\newcommand{\pretr}{\mathrm{pretr}}

\newcommand{\Cat}{{\mathbb{C}}}

\renewcommand{\k}{\Bbbk}

\newcommand{\Assoc}{{{\mathcal{A}ssoc}}}

\newcommand{\pprime}{{\prime\prime}}

\newcommand{\tot}{\mathrm{tot}}

\newcommand{\Coh}{{\mathscr{C}{oh}}}

\newcommand{\Hot}{\mathrm{Hot}}

\newcommand{\RRHom}{{\underline{\RHom}}}

\newcommand{\Graphs}{{\mathbb{G}}}
\newcommand{\dgwu}{{dgwu}}
\newcommand{\dgu}{{dgu}}

\newcommand{\Iso}{\mathrm{Iso}}
\newcommand{\Free}{\mathrm{Free}}
\newcommand{\wu}{\mathrm{wu}}
\newcommand{\abar}{{\bar{a}}}
\newcommand{\Cbar}{{\overline{\mathcal{C}}}}

\newcommand{\pa}{{\mathscr{P}aths}} 
\newcommand{\Paths}{{\mathscr{P}aths}} 
\renewcommand{\C}{\mathcal{C}}
\renewcommand{\D}{\mathcal{D}}
\newcommand{\K}{\mathcal{K'}}
\newcommand{\Y}{\mathcal{Y}}
\newcommand{\X}{\mathcal{X}}

\newcommand{\E}{\mathcal{E}}
\xyoption{2cell}

\title{\sc{A closed model structure on the category of weakly unital dg categories, II}}

\author{\sc{Piergiorgio Panero and Boris Shoikhet}}
\date{}

\begin{document}\maketitle
	{\footnotesize
		\begin{center}{\parbox{4,5in}{{\sc Abstract.}
				In this paper, which is subsequent to our previous paper [PS] (but can be read independently from it), we continue our study of the closed model structure on the category $\Cat_{\dgwu}(\k)$ of small weakly unital dg categories (in the sense of Kontsevich-Soibelman [KS])  over a field $\k$. In [PS], we constructed a closed model structure on the category of weakly unital dg categories, imposing a technical condition on the weakly unital dg categories, saying that $\id_x\cdot \id_x=\id_x$ for any object $x$. Although this condition led us to a great simplification, it was redundant and had to be dropped. 
				
				Here we get rid of this condition, and  provide a closed model structure in  full generality. The new closed model category is as well cofibrantly generated, and it is proven to be Quillen equivalent to the closed model category $\Cat_\dg(\k)$ of (strictly unital) dg categories over $\k$, given by Tabuada [Tab1]. Dropping the condition $\id_x^2=\id_x$ makes the construction of the closed model structure more distant from loc.cit., and requires new constructions. One of them is a pre-triangulated hull of a wu dg category, which in turn is shown to be a wu dg category as well. 
					
				One example of a weakly unital dg category which naturally appears is the bar-cobar resolution of a dg category. We supply this paper with a refinement of the classical bar-cobar resolution of a unital dg category which is strictly unital (appendix B). A similar construction can be applied to constructing a cofibrant resolution in $\Cat_\dgwu(\k)$. 
			}}
		\end{center}
	}


\section{Introduction}
\subsection{}
Many algebraic constructions, dealing with unital dg algebras (or, more generally, unital dg categories) give rise to only weakly ones. 
The simplest example is the bar-cobar resolution $R(A)=\Cobar(\Bar(A))$ of a dg algebra $A$ (resp., a dg category $A$): for unital $A$, $R(A)$ is only weakly unital. Another example which we keep in mind is a further generalisation of the twisted tensor product of small dg categories [Sh], which is supposed to have better homotopical and monoidal properties, and which exists only in the weakly unital context.
It would be beneficial to have a more relaxed (than the closed model category [Tab1] of small dg categories) closed model category, which is formed by small weakly unital dg categories, and which is Quillen equivalent to the closed model category of small dg categories loc.cit. In this paper, we completely solve this problem. 

Among the three definitions of a weakly unital dg category, existing in literature (see [Ly2] for an overview of all of them), the definition given by Kontsevich and Soibelman [KS, Sect. 4.2] seems to be the only one which gives rise to a closed model structure. In this paper, we work with this definition (we recall it in Section \ref{s.wudef}).
\subsection{}
In this paper, we provide a cofibrantly generated Quillen model structure on the category $\Cat_\dgwu(\k)$ of Kontsevich-Soibelman weakly unital dg categories [KS, Sect. 4.2] over a field $\k$ (we recall the definition in Section \ref{s.wudef}). We also establish a Quillen equivallence between the model category $\Cat_\dgwu(\k)$ and the category $\Cat_\dg(\k)$ of unital dg categories over $\k$ (with Tabuada's model structure [Tab1] on it). 

In our previous paper [PS], we constructed a cofibrantly generated Quillen model structure on the category $\Cat^0_{dgwu}(\Bbbk)$ of small Kontsevich-Soibelman weakly unital dg categories\footnote{Our notations here are different from the ones in [PS]. Our $\Cat_\dgwu(\k)$ here was denoted by $\mathrm{Cat}^\prime_\dgwu(\k)$ in [PS], and  our $\Cat^0_\dgwu(\k)$ was  $\mathrm{Cat}_\dgwu(\k)$ in [PS].}  over a field $\Bbbk$ with an extra condition saying that $\id_X\circ \id_X=\id_X$ for any object $X$.  We also proved that there was a Quillen equivalence $L\colon \Cat^0_\dgwu(\k)\rightleftarrows \Cat_\dg(\k): R$, where $\Cat_\dg(\k)$ is the closed model structure on the category of small (unital) dg categories over $\k$. The condition $\id_X\circ\id_X=\id_X$ had a technical nature and was imposed  to avoid some difficulties in the proofs we experienced for the general case of $\Cat_\dgwu(\k)$. However, this condition is rather artificial and is not fulfilled in any reasonable example, so the next step is to get rid of it. We complete it here. 

Our results are counterparts of the main results of [PS] in this general setting, though the proofs need some essentially new ideas and constructions. Among them, we  draw reader's attention to the pre-triangulated hull of a weakly unital dg category (Section 3), and a weakly unital replacement of the Kontsevich dg category $\mathcal{K}$ (Section \ref{sectionkcat}).

We prove
\begin{theorem}[proven in Theorem \ref{bigtheorem}]\label{theoremintro1}
For a field $\k$, there is a cofibrantly generated Quillen model structure on $\Cat_\dgwu(\k)$.
\end{theorem}

\begin{theorem}[proven in Prop. \ref{propqp} and Theorem \ref{quieq}]\label{theoremintro2}
There is 
a Quillen equivalence
$$
L\colon \Cat_\dgwu(\k)\rightleftarrows \Cat_\dg(\k): R
$$
where $\Cat_\dg(\k)$ is endowed with the Tabuada closed model structure [Tab1]. 
\end{theorem}

\comment
Our second main result is an application of the developed theory, which actually motivated the project. 
We provide a direct approach to computation of the $\Hom$-sets $\Hot(C,D)$ (correspondingly, $\Hom$-simplicial sets $\Hot^\Delta(C,D)$) in the homotopy category $\Hot(\k)$ of the category $\Cat_\dg(\k)$. 

Let us outline the difficulties one meets here when one tries to compute these sets. It is well-known that, in order  to compute $\Hom$-sets in the homotopy category of a closed model category, one has to replace the first argument by its cofibrant resolution and the second argument by its fibrant resolution, and take the quotient-set of the corresponding $\Hom$-set by the equivalence relation, defined by maps of the first argument to a path object of the second argument. In $\Cat_\dg(\k)$, $\k$ a field, any object is fibrant, and Tabuada's paper [Tab2] provides a nice choice of the path object. Furthermore, a natural cofibrant replacement of a (unital) dg category $C$ is $\Cobar(\Bar(C))$, and here one meets the main trouble. The key point is that the dg category $\Cobar(\Bar(C))$ fails to be a unital dg category, though it is a KS weakly unital dg category. In this respect, the Quillen equivalence of Theorem \ref{theoremintro1} makes it possible to compute $\Hom$-sets in the homotopy category of $\Cat_\dgwu(\k)$. The next point is that $\Cobar(\Bar(C))$ ifails to be cofibrant in $\Cat_\dgwu(\k)$. We provide a new construction, in which adjust the cobar-construction to the weakly unital case. It gives rise to a cofibrant weakly unital dg category $R(C)=\FCobar(\Bar(C))$, which is a ``fattened version'' of $\Cobar(\Bar(C))$, and still is a resolution of $C$. We make use of this category $R(C)$ for our computation.

It gives rise to the following description of $\Hot(C,D)$. 

Recall (see Appendix A) that an $A_\infty$ functor $F\colon C\to D$ with Taylor coefficiants $\{F_i\}_{i\ge 1}$ is called {\it weakly unital}  if $F_1(\id_X)=\id_{F(X)}$ for any object $X\in C$.

Recall that, for two $A_\infty$ functors $F,G\colon C\to D$, one defines the complex of $A_\infty$ natural transformations, denoted by $\Coh_{A_\infty}(C,D)(F,G)$. 
An $A_\infty$ natural transformation $\Theta\colon F\Rightarrow G$ of degree $d$ is given by its Taylor components $\{\Theta_n\}_{n\ge 0}$, where $\Theta_n\in \Hoch^{n}(C,D^{}(F(-),G(=)))$ of total degree $d$. 
\endcomment

\comment
We call an $A_\infty$ natural transformation {\it 1-truncated} if $\Theta_i=0$ for $i\ge 2$. When $F,G$ are dg functors, a 1-truncated {\it closed} $A_\infty$ functor $\Theta$ of degree $0$ is given by a family of closed maps $\Theta_0(X)=D^0(F(X),G(X))_{X\in C}$ and a family of maps $\{\Theta_1(f)\in D^{-1}(F(X),G(Y))\}_{f\in C(X,Y)}$ such that
$$
d\Theta_1(f)- \Theta_1(df)=G(f)\circ \Theta_0(X)-\Theta_0(Y)\circ F(f)
$$
and
$$
\Theta_1(f_2)\circ F(f_1)-\Theta_1(f_2f_1)+G(f_2)\circ \Theta_1(f_1)=0
$$

\begin{theorem}[proven as ]\label{theoremintro2}
Let $C,D$ be two unital dg categories. Then the set $\Hot(C,D)$ is the quotient-set
$$
\{\text{weakly unital }A_\infty\text{ functors }C\to D\}/\sim
$$
where $F\sim G$ iff there is a closed degree 0 $A_\infty$ natural transformation $\Theta\colon F\Rightarrow G$, for which the closed degree 0 morphisms $\{\Theta_0(X)\in D(F(X),G(X))\}_{X\in C}$ are homotopy equivalences.\footnote{that is, induce an isomorphism in $H^0(D)$}
\end{theorem}

A similar description for the simplicial set $\Hot^\Delta(C,D)$ is given in Theorem ???.

\vspace{2mm}

To the best of our knowledge, the existing computation is based on the description by Toen [To] of the homotopy category $\Hot(\k)$ via quasi-functors.
Faonte's recent paper [Fa] proves that Toen's internal $\Hom$ dg category $\RRHom(C,D)$ is quasi-equivalent to the dg category $\Coh_{A_\infty}(C,D)$ whose objects are strongly unital $A_\infty$ functors $C\to D$, and whose morphisms are $A_\infty$ natural transformations. Taken together, these two results can be used to get a description similar to ours'.

Indeed, one can show that $\Hot(C,D)=\pi_0(N(\Iso(H^0(\RRHom(C,D)))))$ where $\Iso(-)$ stand for the ordinary category whose morphisms are degree 0 isomorphisms of its argument, and $N(-)$ stands for the nerve of an ordinary category.

However, [To] is based on far more involved homotopy theory, and, in our opinion, it is beneficial to have a direct approach. 

\endcomment
\subsection{} 
Here we outline in more detail our results and the organisation of the paper.  

In Section 2, we recall the main definitions from [PS], and the results from loc.cit., used here. These are small completeness and small cocompleteness of $\Cat_\dgwu(\k)$, and the monadic description of it. Namely, there is a monad on the category of unital dg graphs $\Graphs_\dgu(\k)$, such as the algebras over this monad are the weakly unital dg categories. Moreover, this monad is associated with a non-symmetric dg operad, called $\mathcal{O}^\prime$. Here one of the main results is cited from loc.cit. as Theorem \ref{compop} here. It says that a natural map gives a quasi-isomorphism of dg operads $\mathcal{O}^\prime\to \Assoc_+$ to the operad of unital dg algebras. This computation is essentially used in the proof that the Quillen pair of Prop. \ref{propqp} is a Quillen equivalence. This Section makes it possible reading this paper independently on [PS]. 

In Section \ref{s.def}, we prove that to any weakly unital dg-category $C$ we can functorially associate the pretriangulated hull of it $C^{\pretr}$, and this non unital dg-category is in turn  weakly unital. 
Here the main point, which makes the problem non-trivial, is that, for a closed morphism $f$ in $C(x,y)$, the identity morphism $\id_{\Cone(f)}$ within the ``naive'' definition, fails to be closed, if $\id_y\circ f\ne f\circ \id_x$. 
This construction comes out naturally when one perturbs the naive construction making $\id_{\Cone(f)}$ closed.  This tool allows us to work with cones of any closed morphism in $C$, and these are  key points for crucial  Lemma \ref{kcontr} and Lemma \ref{Fibinj}.

In Section \ref{s.mc}, we construct a closed model structure on $\Cat_\dgwu(\k)$. Here the strategy is rather closed to [Tab1] and [PS], though the generating cofibrations $I$ and generating acyclic cofibrations $J$ should be re-defined to the weakly unital context, making them ``fattened''. This ``fattening'' goes straightforwardly for all generating (acyclic) cofibrations except to the one in which the Kontsevich dg category $\mathcal{K}$ is involved. The refined wu version of $\mathcal{K}$ is not straightforward (see Section \ref{sectionkcat}, and \eqref{eqrelk}). It is a dg category which classifies a homotopy equivalence in a wu dg category (Lemma \ref{lem.kon}). We also prove crucial Lemma \ref{kcontr}, connecting $\mathcal{K}$ and weakly unital pre-triangulated dg hull. This Lemmas are used in the proof of key-lemma Lemma \ref{Fibinj}, which completes the story.

In Section \ref{s.quieq}, we establish a Quillen equivalence between $\Cat_\dgwu(\k)$ and $\Cat_\dg(\k)$. Here we essentially use Theorem \ref{compop}, saying that the dg operad $\mathcal{O}^\prime$ is quasi-isomorphic $\Assoc_+$. 

Finally, Section 6 contains a computational proof of Proposition \ref{propph}. 

We supply the paper with three Appendices. In Appendix A, we just fix signs in the identities on the Taylor coefficients of an $A_\infty$ morphism. Appendix B contains a (seemingly new but elementary) construction of a unital cofibrant resolution of a unital dg algebra, easily generalizing to the case of dg categories. Appendix C details a computation of particular coequalizers, following from the general discussion in [PS]; this computation is used in the proof of Theorem \ref{quieq}.

\subsection*{\sc Acknowledgements}
We are thankful to Bernhard Keller for his encouragement. 
The work of both authors was partially supported by the FWO Research Project Nr. G060118N.
The work of B.Sh. was partially supported by the HSE University Basic Research Program and Russian Academic Excellence Project '5-100'.

\section{\sc Weakly unital dg categories}
\subsection{}\label{s.wudef}
Recall the definition of a weakly unital dg category [KS, 4.2].

Let $A$ be a non unital dg category. Denote by $A\oplus \Bbbk_{A}$ the strictly unital dg category where $Ob(A\oplus \Bbbk_{A}) = Ob(A)$ and
\[ \Hom_{A\oplus \Bbbk_{A}}(x,y) = \begin{cases}  \Hom_{A}(x,y) &  \text{ if } x \neq y\\ 
 \Hom_{A}(x,x)\oplus \Bbbk 1_x &  \text{ if } x=y. \end{cases} 
              \]
One has a natural imbedding $i: A \rightarrow A\oplus \Bbbk_{A}$, sending $x$ to $x$, and $f \in A(x,x)$ to the pair $(f,0) \in (A\oplus \Bbbk_{A})(x,x)$. We denote by $1_x$ the generator of $\k_x$.
\begin{defn}\label{wudef}{\rm
A \emph{weakly unital dg category} $A$ over $\Bbbk$ is a non-unital dg category $A$ over $\Bbbk$ with a distinguished closed element $\id_x \in A^{0}(x,x)$ for any object $x$ in $A$, such that there exists an $A_{\infty}$-functor $p: A\oplus \Bbbk_{A} \rightarrow A$ which is the identity on the objects, such that $p \circ i = \id_{A}$, $p_1(1_x) = \id_x, \forall x \in Ob(A)$, and $p_n(f_1,\dots,f_n)=0$ for $n\ge 2$ and all $f_i$ morphisms in the image $i(A)$.}
\end{defn}

\begin{defn}\label{wudeffunctor}{\rm 
Let $A, C$ be two weakly unital dg categories, with the structure maps $p^{A}: A\oplus \Bbbk_{A} \rightarrow A$ and $p^{C}: C\oplus \Bbbk_{C} \rightarrow C$.   
A \emph{weakly unital dg functor} $F: A \rightarrow C$ is a non unital dg functor $F\colon A\to C$ such that the following diagram commutes:
\begin{equation}
\begin{tikzpicture}[baseline= (a).base]
\node[scale=1.3] (a) at (5,5){
	\begin{tikzcd}
	A\oplus \Bbbk_{A} \arrow[d, "p^{A}"'] \arrow[r, "F\oplus\Bbbk_{F} "] & C\oplus \Bbbk_{C} \arrow[d, "p^{C}"]     \\
	A \arrow[r, "F"]   & C                                                   
	\end{tikzcd}
};
\end{tikzpicture}
\end{equation}
}
\end{defn}   

In this way, we define the category $\Cat_\dgwu(\k)$. Its full subcategory, for which $\id_x\circ \id_x=\id_x$ for any object $x$, is denoted by $\Cat_\dgwu^0(\k)$.

It follows from the definition that:
\begin{equation}
\begin{aligned}
F(\id_x) &= \id_{F(x)} &\forall x \in \Ob(A) \\
F(p^{A}_n(f_1, .., f_n)) &= p^{\C}_n(F(f_1), .., F(f_n))  & f_i \in A\oplus\k_A, i=1\dots n
\end{aligned}
\end{equation}    

\begin{example}\label{exwusu}
	\rm{
Let $C$ be a strictly unital dg category. We denote by $i(C)$ the weakly unital dg category, for which the $A_\infty$ functor $p\colon C\oplus\k_C\to C$ is a natural dg functor, so that all $p_n^{i(C)}, n\ge 2$ are equal to 0. In this way we get a fully-faithful embedding $i\colon \Cat_\dg(\k)\to\Cat_\dgwu(\k)$.	
}
\end{example}

This definition is due to Kontsevich-Soibelman loc.cit. 
There are two other definitions, due to Lyubashenko and Fukaya, correspondingly. We refer the reader to [Ly], [LyMa]  for discussion of all three definitions, and to [COS] for interplay between them.

An advantage of the Kontsevich-Soibelman definition is that the category of small KS wu dg categories carries a closed model structure, unlike for the two other definitions. Such closed model structure was provided in [PS] for $\Cat_\dgwu^0(\k)$, and is done here for the general case of  $\Cat_\dgwu(\k)$, see Theorem \ref{bigtheorem} below. 

We recall the definition of an  $A_\infty$ morphism and recall the signs in the $A_\infty$ identities in Appendix A, see \eqref{ainftysigns}.

Here we write down, for further reference, formulas for $dp_2$ and $dp_3$, coming from \eqref{ainftysigns}.

\begin{equation}\label{eqp2}
dp_2(f,1_x)+p_2(df,1_x)=f-f\circ \id_x, \hspace{4mm}dp_2(1_x,f)+p_2(1_x,df)=f-\id_x\circ f
\end{equation}  

\begin{equation}\label{eqp3}
\begin{aligned}
\ & dp_3(f,g,1_x)-(-1)^{|g|}p_3(df,g,1_x)-p_3(f,dg,1_x)=f\circ p_2(g,1_x)-p_2(f\circ g,1_x)\\
&dp_3(f,1_x,g)-(-1)^{|g|}p_3(df,1_x,g)-p_3(f,1_x,dg)=-(-1)^{|g|} p_2(f,1_x)\circ g+ f\circ p_2(1_x,g)\\
&dp_3(1_x,f,g)-(-1)^{|g|}p_3(1_x,df,g)-p_3(1_x,f,dg)=-(-1)^{|g|}p_2(1_x,f)\circ g+p_2(1_x,f\circ g)\\
&dp_3(1_x,1_x,f)-p_3(1_x,1_x,df)=\id_x\circ p_2(1_x,f)-(-1)^{|f|}p_2(1_x,1_x)\circ f \\
&dp_3(1_x,f,1_x)-p_3(1_x,df,1_x)=\id_x\circ p_2(f,1_x)-p_2(1_x,f)\circ \id_x-p_2(f,1_x)+p_2(1_x,f)\\
&dp_3(f,1_x,1_x)-p_3(df,1_x,1_x)=f\circ p_2(1_x,1_x)   -p_2(f,1_x)\circ \id_x  \\
&dp_3(1_x,1_x,1_x)=\id_x\circ p_2(1_x,1_x)-p_2(1_x,1_x)\circ \id_x
\end{aligned}
\end{equation}

\subsection{\sc The monadic structure and the dg operad $\mathcal{O}^\prime$}
Here we recall some results of [PS] which we use in this paper. 

With the concept of a weakly unital dg category is associated a monad acting on the category of oriented dg graphs such that a weakly unital dg category amounts to the same thing as an algebra over this monad. This description is used in loc.cit. to prove that the category $\Cat_\dgwu(\k)$ is small complete and small cocomplete. The monad itself is defined via a non-symmetric dg operad called $\mathcal{O}^\prime$. Here we briefly recall the corresponding definitions and results. 

An oriented dg graph $\Gamma$ over $\k$ is given by a {\it set} $V_\Gamma$ of vertices, and a complex $\Gamma(x,y)$ for any ordered pair $x,y\in V_\Gamma$. A morphism $F\colon \Gamma_1\to\Gamma_2$ is given by a map of sets $F_V\colon V_{\Gamma_1}\to V_{\Gamma_2}$, and by a map of complexes $F_E\colon \Gamma_1(x,y)\to \Gamma_2(F_V(x),F_V(y))$, for any $x,y\in V_{\Gamma_1}$. We denote by $\Graphs_\dg(\k)$ the category of the dg graphs over $\k$. A {\it unital} dg graph $\Gamma$ is an oriented dg graph such that there is an element $\id_x\in \Gamma(x,x)$, closed of degree 0, for any $x\in V_\Gamma$. A map of unital dg graphs is a map $F$ of the underlying dg graphs such that $F(\id_x)=\id_{F(x)}$, for any $x\in V_\Gamma$. We denote by $\Graphs_\dgu(\k)$ the category of unital dg graphs over $\k$. 

There is a natural forgetful functor $U\colon \Cat_{dgwu}(\k)\to \Graphs_\dgu(\k)$, where $U(C)$ is a graph $\Gamma$ with $V_\Gamma=\Ob(C)$, and $\Gamma(x,y)=C(x,y)$.

This functor admits a left adjoint $F\colon \Graphs_\dgu(\k)\to \Cat_\dgwu(\k)$. It is constructed via a dg operad $\mathcal{O}^\prime$. 

Consider the non-$\Sigma$ the dg operad $\mathcal{O}$ define as the quotient-operad of the free operad generated by the composition operations:
\begin{itemize}
	\item[(a)] the composition operation $m\in\mathcal{O}^\prime(2)^0$
	\item[(b)] $p_{n; i_1,\
		\dots,i_k}\in \mathcal{O}^\prime(n-k)^{-n+1}$, $0\le k\le n$, $1\le i_1<i_2<\dots<i_k\le n$, with the following meaning:
	For a weakly unital dg category $C$, the operation $p_{n; i_1,\dots,i_k}(f_1,\dots,f_{n-k})$ is defined as 
	\begin{equation}\label{relop}
	p_n\big(f_1,\dots,f_{i_1-1}, \underset{i_1}{1_{x_1}},f_{i_1},\dots,f_{i_2-2},\underset{i_2}{1_{x_2}},f_{i_2-1},\dots,f_{i_3-3},\underset{i_3}{1_{x_3}},\dots\dots,\underset{i_k}{1_{x_k}},f_{i_k-k+1},\dots,f_{n-k}\big)
	\end{equation}
	where by $1_{x_i}$s are denoted the morphisms $1_{x_i}\in \k_C$ for the corresponding objects $x_i\in C$.
\end{itemize}
by the following relations:
\begin{equation}\label{orel}
\begin{aligned}
\ &(i) \text{ the associativity of  $m$, and $dm=0$}\\
&(ii)\  p_{n; i_1,\dots,i_k}=0 \text{ if $k=0$}\\
&(iii)\  p_{1; -}=\id\\
&(iv)\  \text{the $A_\infty$ morphism relation for }
dp_{n; i_1,\dots,i_k}
\end{aligned}
\end{equation}
\comment
Note that (ii) formally follows from the part of (iii), saying that $p_{n;1,2,\dots,n}=0$, and (v).
\endcomment
We use the notation $j=p_{1,1}$, the degree zero 0-ary operation generating the weak unit. It follows from (iv) that $dj=0$. The reader is referred to [PS, Sect. 1.2.3] for the precise form of the relation in (iv). It expresses the relations like \eqref{eqp2}, \eqref{eqp3} in the operadic terms, using the correspondence \eqref{relop}.

Morally, the dg operad $\mathcal{O}^\prime$ comprises all universal operations one can define on a weakly unital dg category.

Denote by $\Assoc_+$ the operad of unital associative $\k$-algebras.
We proved the following theorem:
\begin{theorem}\label{compop}
The natural map of dg operads $\mathcal{O}^\prime\to\Assoc_+$,
sending all $p_{n; i_1,\dots,i_k}$, $n\ge 2$, to 0, sending $j=p_{1;1}$ to 	the 0-ary unit generating operation, and sending $m$ to $m$, is a quasi-isomorphism.
\end{theorem}
The proof is a rather long and tricky computation with different spectral sequences, see [PS, Theorem 1.13]. 

\qed

The left adjoint functor $F\colon \Graphs_\dgu(\k)\to\Cat_\dgwu(\k)$ is defined in two steps, as follows. Given a unital dg graph $\Gamma$, consider the free $\mathcal{O}^\prime$-algebra $T_{\mathcal{O}^\prime}(\Gamma)$, generated by $\Gamma$ (see [PS, Sect. 1.2.3]). It is a weakly unital dg category with objects $V_\Gamma$. The 0-ary operation $j$ generates an element $j_x\in T_{\mathcal{O}^\prime}(x,x)$, for any $x\in V_\Gamma$. 
After that, define $F(\Gamma)$ as the dg quotient-category
\begin{equation}
F(\Gamma)=T_{\mathcal{O}^\prime}(\Gamma)/(j_x-\id_x, x\in V_\Gamma)
\end{equation}
In this way, we identify $\id_x\in \Gamma(x,x)$ with the ``weak unit'' $j_x$ generated by $\mathcal{O}^\prime$.

One has:
\begin{prop}
There is an adjunction:
\begin{equation}
\Cat_\dgwu(F(\Gamma),C)\simeq \Graphs_\dgu(\Gamma, U(C))
\end{equation}
\end{prop}

\qed

Note that for $\Gamma$ a (non-unital) dg graph, one defines a unital dg graph $\Gamma_+$, formally adding $\k_x$ to $\Gamma(x,x)$, for any $x\in V_\Gamma$. Then

$$
F(\Gamma_+)\simeq T_{\mathcal{O}^\prime}(\Gamma)
$$

\section{\sc The pretriangulated hull of a weakly unital dg category}\label{s.def}
Recall that the {\it pretriangulated hull} of a dg category $C$ was introduced by Bondal-Kapranov [BK] (see also [Dr, 2.4, Remark]).
A dg category $C$ is {\it pretriangulated} if $H^0(C)$ is triangulated. Explicitly, it means that the functors $Z\to \underline{\Hom}(C(Z,X)\xrightarrow{f}C(Z,Y))$ and $Z\to C(Z,X)[n]$ defined 
for any closed morphism $f\colon X\to Y$ in $C$ and for any object $X\in C$, $n\in\mathbb{Z}$, correspondingly, are representable. In this case, the representing objects are $\Cone(f)$ and $X[n]$. 

The pretriangulated hull $C^\pretr$ of a dg category $C$ is a pretriangulated dg category with a dg functor $C\to C^\pretr$, which is universal for dg functors from $C$ to pretriangulated dg categories [BK]. 

Explicitly, it is constructed as follows. An object of $C^\pretr$ is a ``one-sided twisted complexes'', which are formal finite sums $(\oplus_{i=1}^n X_i[r_i]), q)$ where $q$ has components $q_{ij}\in C^{r_i-r_j+1}(X_i,X_j)$, which are zero for $i\ge j$, such that $dq+q^2=0$. Let $X=(\oplus X_i[r_i], q)$, $X^\prime=(\oplus X_i^\prime[r_i^\prime], q^\prime)$ be two objects of $C^\pretr$, a morphism $\phi\in C^\pretr(X,X^\prime)$ of degree $k$ is defined as the collection $\phi_{ij}\colon X_i[r_i]\to X^\prime_j[r_j^\prime]$ of degree $d$ (in general, non-zero for any $i,j$), and 
$d\phi=d_C\phi+q^\prime\circ \phi-(-1)^k\phi\circ q$. The composition is defined as the matrix product. 

The dg functor $C\to C^\pretr$ is defined on objects as $X\mapsto (X[0],q=0)$, and on morphisms accordingly. We recall that $\Cone(f)\in C^\pretr$, $f\colon X\to Y$ a closed morphism in $C$, is defined as $Cone(f)=(X\oplus Y[-1], q=f)$ (that is, $q_{12}=f$, $q_{11}=q_{22}=q_{21}=0$).

We want to define the pretriangulated hull of a {\it weakly unital} dg category, which is a weakly unital dg category as well. If we just repeated the definition given above, we would experience the following problem. Let $f\in C(x,y)$ be a closed morphism, we do {\it not} assume that $f\circ \id_x=f$ or $\id_y\circ f=f$. 
We $\Cone(f)=(X\oplus Y[-1],f)$, there should be a weak identity morphism $\id_{\Cone(f)}$, which is a closed morphism of degree 0. A natural candidate is given by $\id_X\colon X\to X, \id_Y\colon Y\to Y$. But then $d(\id_{\Cone(f)})=f\circ \id_X-\id_Y\circ f\ne 0$.

We remedy this problem as follows. 
For a closed morphism $f$ in $C$, define $\Cone(f)=(X\oplus Y[1], f))$ as above, but re-define $\id_\Cone(f)$. Namely, {\it define} $\id_\Cone(f)$ as having 3 non-zero components:
\begin{equation}\label{eqidcone}
\begin{aligned}
\ &\id_{\Cone(f)}=(\id_X,\id_Y,\varepsilon\in C^0(X,Y[1]))\\
&\hspace{7mm}\text{where  }\varepsilon=p_2(f,1_x)-p_2(1_y,f)
\end{aligned}
\end{equation}
where $p_2$ is the second Taylor component of the $A_\infty$ morphism $p\colon C\oplus \k_C\to C$, see Definition \ref{wudef}.

Then one has:
\begin{equation}
d(\id_{\Cone(f)})=f\circ \id_X-\id_Y\circ f+d\varepsilon=0
\end{equation}
(Recall that $dp_2(f,1_x))=p_1(f)-f\circ \id_X=f-f\circ \id_X$, and similarly for $dp_2(1_y,f)$). 

Thus, at the first step we define, inspired by this example, the identity morphism $\id_X$, for $X=(\oplus X_i[r_i], q_{ij})$, and check that $d(\id_X)=0$. After that, we construct an $A_\infty$ functor $P\colon C^\pretr\oplus\k_{C^\pretr}\to C^\pretr$, making $C^\pretr$ a weakly unital dg category. 

(We denote by $p$ the structure $A_\infty$ functor for $C$, and by $P$ the structure $A_\infty$ functor for $C^\pretr$). 

\begin{defn}[Pretriangulated hull of a wudg category, I]\label{defwuph1}
	{\rm Let $C$ be a wudg category. We define the underlying non-unital dg category of pretriangulated hull of ${C}$ as in the strictly unital case:
	\begin{itemize}
	\item[(a)] objects are formal expressions $(\bigoplus_{i=1}^{n} X_i[r_i], q_{ij})$, where $X_i \in \mathcal{C}$, $r_i \in \mathbb{Z}$, $q_{ij} \in \mathcal{C}^{1+r_j-r_i}(X_i,X_j)=\mathcal{C}^1(X_i[r_i],X_j[r_j])$ such that $q_{ij} = 0$ if $i \geq j$ and $dq + q\circ q = 0$,
	\item[(b)] the space of degree $k$ morphisms $C^\pretr(X,X^\prime)$, for $X=(X_i[r_i],q_{ij}), X^\prime=(X^\prime[r_i^\prime],q^\prime_{ij})$, is defined as the space of matrices
	$\phi=(\phi_{ij}\colon C^k(X_i[r_i]\to X^\prime_j[r_j^\prime]))$, 
	the composition is matrix multiplication and the differential is $d\phi := d_{C}\phi + q'\circ \phi - (-1)^{k} \phi \circ q$.
	\end{itemize}
}
\end{defn}

Now we define, for any object $X\in C^\pretr$, an ``identity'' morphism $\id_X$ (which is required to be a closed morphism of degree 0), and construct an $A_\infty$ morphism $P\colon C^\pretr\oplus\k_{C^\pretr}\to C^\pretr$, making it a weakly unital dg category. In fact, we start with $A_\infty$ morphism $P$, then $\id_X:=P_1(1_X)$. 

Let $X^0,\dots, X^n$ be objects of $C^\pretr$, and let $\phi^i\colon X^{i-1}\to X^i$ be either a morphism in $C^\pretr$ or $1_{X^{i-1}}$ (in which case $X^i=X^{i-1}$).

We are going to define $P_n(\phi^n,\dots,\phi^1)$. Let us introduce some notations. We visualize the string 
$$
X^0\xrightarrow{\phi^1}X^1\xrightarrow{\phi^2}X^2\to\dots\xrightarrow{\phi^n}X^n
$$
as a planar diagram whose horizontal arrows are $q^i_{k\ell}$, where $X^i=(\oplus X^i_k[r^i_k], q^i_{k\ell}$), and whose other arrows are the components of $\phi^i$, $i=1,\dots,n$, see \eqref{primodiagramm}. 

We refer to the arrows $q^i_{k\ell}$ as {\it horizontal}, and the other arrows, called {\it essential}, are either components of $\phi^i$s or morphisms $1_X$ for some $X\in\mathcal{C}$.

Now we associate to any couple $(X^0_a, X^n_b)$ of starting and ending objects, a set $\pa_{ab}$ of \emph{\textbf{all}} the possible paths from $X^0_a$ to $X^n_b$, see  \ref{primodiagramm}, \eqref{secondodiagramm}. \\

By definition, a {\it path} $\kappa\in \Paths_{ab}$ is a sequence of arrows $\kappa=(\kappa_1,\dots, \kappa_\ell)$, either horizontal or essential, such that (a) for any $1\le s\le n$ there is exactly 1 essential arrow which is a components of $\phi^s$, and these $n$ essential arrows stand respecting the order, (b) the arrows between two successive essential arrows, which are components of $\phi^s$ and $\phi^{s+1}$ (here $\phi^s=1_X$ is allowed), are horizontal arrows in $X^{s}$, which form a composable chain (there are allowed more than 1 arrows in this chain), (c) the first arrow starts at $X^0_a$, and the last one ends at $X^n_b$. It follows in particular that a path is represented by a composable chain of arrows.\\

If some $\phi^i=1_{X^{i-1}}$, the corresponding arrow in $\phi^i_{k\ell}\colon X^{i-1}_k\to X^{i-1}_\ell$ is defined as $1_{X^{i-1}_k}$ for $k=\ell$, and 0 otherwise. \\

For example, in \eqref{primodiagramm} the sequence  
$(q_{i\ell}, q_{\ell m}, \phi_{m\ell}, q^{'}_{\ell m}, \phi^{'}_{mm}, q^{''}_{mj})$
 is a path, and in \eqref{secondodiagramm} the sequence $(\phi_{ii}, q^{'}_{i\ell}, q^{'}_{\ell k}, \phi^{'}_{ki}, q^{''}_{ij})$ is a path (here for both diagrams $n=2$). 

\begin{equation}\label{primodiagramm}
\begin{tikzpicture}[baseline= (a).base]
\node[scale=1.0] (a) at (0,0){
\begin{tikzcd}
... \arrow[r, dotted] & X^0_i \arrow[r, "q_{i\ell}"] \arrow[rd, dotted] \arrow[d, dotted] \arrow[rrd, dotted] & X^0_{l} \arrow[r, "q_{\ell m}"] \arrow[d, dotted] \arrow[ld, dotted]                    & X^0_{m} \arrow[ld, "\phi_{ml}"'] \arrow[r, dotted] \arrow[rd, dotted] \arrow[d, dotted] & .. X^0_{j} \arrow[d, dotted] \arrow[ld, dotted] \arrow[lld, dotted] \arrow[r, dotted] & ...\\
... \arrow[r, dotted] & X^1_{i} \arrow[r, dotted] \arrow[rd, dotted] \arrow[d, dotted]                  & X^1_{l} \arrow[r, "q^{'}_{\ell m}"'] \arrow[rd, dotted] \arrow[d, dotted] \arrow[ld, dotted] & X^1_{m} \arrow[d, "\phi^{'}_{mm}"] \arrow[r, dotted] \arrow[ld, dotted] \arrow[rd, dotted] & X^1_{j} \arrow[d, dotted] \arrow[ld, dotted]  \arrow[r, dotted] & ... \\
... \arrow[r, dotted] & X^2_{i} \arrow[r, dotted]                                                       & X^2_{l} \arrow[r, dotted]                                                      & X^2_{m} \arrow[r, "q^{''}_{mj}"']                                                          & X^2_j  \arrow[r, dotted] & ...\\                                                    
\end{tikzcd}
};
\end{tikzpicture}
\end{equation}
\begin{equation}\label{secondodiagramm}
\begin{tikzpicture}[baseline= (a).base]
\node[scale=1.0] (a) at (5,5){
\begin{tikzcd}
... \arrow[r, dotted] & X^0_i \arrow[d, "\phi_{ii}"] \arrow[r, dotted] \arrow[rd, dotted] \arrow[rd, dotted] \arrow[rrd, dotted] & X^0_l \arrow[r, dotted] \arrow[rd, dotted] \arrow[d, dotted] \arrow[ld, dotted]                            & X^0_m \arrow[r, dotted] \arrow[rd, dotted] \arrow[ld, dotted] \arrow[d, dotted] & X^0_j \arrow[d, dotted] \arrow[ld, dotted]  \arrow[r, dotted] & ... \\
... \arrow[r, dotted] & X^1_i \arrow[r, "q^{'}_{i\ell}"] \arrow[rd, dotted] \arrow[d, dotted]                                        & X^1_l \arrow[rr, "q^{'}_{\ell k}", bend left] \arrow[r, dotted] \arrow[d, dotted] \arrow[ld, dotted] \arrow[rd, dotted] & X^1_m \arrow[r, dotted] \arrow[rd, dotted] \arrow[d, dotted] \arrow[ld, dotted] & X^1_j \arrow[llld, "\phi^{'}_{ki}"'] \arrow[d, dotted] \arrow[ld, dotted] \arrow[r, dotted] & ... \\
... \arrow[r, dotted] & X^2_i \arrow[rrr, "q^{''}_{ij}", bend right] \arrow[r, dotted]                                             & X^2_l \arrow[r, dotted]                                                                                    & X^2_m \arrow[r, dotted]                                                         & X^2_j  \arrow[r, dotted] & ...                                 \\                    
\end{tikzcd}
};
\end{tikzpicture}
\end{equation}

Below we assume for $n\ge 2$ that at least one of morphisms $\phi^i\colon X^{i-1}\to X^i$ is $1_{X^{i-1}}$; otherwise, \eqref{bigP} below gives 0. \\

Define 
\begin{equation}\label{bigP}
P_n^{ij}(\phi_n, .., \phi_1):= \sum_{\kappa \in \pa_{ij}} (-1)^{|\kappa|} p_l(\kappa_l, .., \kappa_1)
\end{equation}
(Recall that $p$ denotes the structure $A_\infty$ morphism for $C$). \\

To define the integral number $|\kappa|$, we introduce some notations. Let $\kappa=(\kappa_1,\dots,\kappa_\ell)$, and let the $n$ arrows $\kappa_{d_1},\dots,\kappa_{d_n}$ be essential. 
Assume that $\kappa_{d_s}$ is an arrow in $\mathcal{C}(X^{s-1}_{a_s}[r^s], X^s_{b_s}[r^{s+1}])$.  Define $t_s=
r^{s+1}-r^s$ (we set $t_s=0$ if $\kappa_{d_s}=1_X$). \\

The integral number $|\kappa|$ is given by
\begin{equation}
|\kappa|=\sum_{s=1}^{n} (\deg\phi^{s}+t_s+1)N_s
\end{equation}
where $N_s$ is the number of the horizontal arrows standing {\it leftwards} to the $s$-th essential arrow $\kappa_{d_s}$ {\it in the sequence} $(\kappa_1,\dots,\kappa_\ell)$.

\begin{lemma}\label{lemmaph}
The maps $P_n^{ij}(\phi_n,\dots,\phi_1)$ are homogeneous of degree $\sum\deg \phi_i-n+1$. Thus, they are the components of a morphism $P_n(\phi_n,\dots,\phi_1)\colon X^0\to X^n$ of degree $\sum\deg\phi_i-n+1$ in the category $C^\pretr$. 
\end{lemma}
\begin{proof}
	Let $\kappa=(\kappa_1,\dots,\kappa_\ell)\in \Paths_{ij}$, we have to compute the degree of $p_\ell(\kappa_n,\dots,\kappa_1)$. One has:
	$$
	\deg p_\ell(\kappa_\ell,\dots,\kappa_1)=\sum_{r=1}^\ell\deg\kappa_r -\ell+1
	$$
	Among $\kappa_1,\dots,\kappa_\ell$ exactly $n$ arrows are $\phi^i_{st}$s, the remaining $\ell-n$ are $q^i_{st}$ and have degree 1 in $C^\pretr$. On the other hand, $\deg \phi^i_{st}=\deg\phi^i$ and does not depend on $s,t$.  Therefore,
	$$
	\deg p_\ell(\kappa_\ell,\dots,\kappa_1)=\sum_{\kappa_r\ne q^j_{st}}\deg \kappa_r -n+1=\sum_{r=1}^n\deg \phi_r-n+1
	$$
\end{proof}

\begin{prop}\label{propph}
Let $C$ be a weakly unital dg category, $C^\pretr$ the non-unital dg category from Definition \ref{defwuph1}. 
Taken for all $n\ge 1$ and all $\phi_1,\dots,\phi_n$, the morphisms $P_n(\phi_n,\dots,\phi_1)$ are the Taylor components of an $A_\infty$ morphism $P\colon C^\pretr\oplus\k_{C^\pretr}\to C^\pretr$,  making $C^\pretr$ a weakly unital dg category, with $\id_X:=P_1(1_X)$.
\end{prop} 

We prove this Proposition in Section 6. 

\begin{defn}[Pretriangulated hull of a wudg category, II]\label{defwuph2}
	{\rm The pretriangulated hull $C^\pretr$ of a weakly unital dg category $C$ is the non-unital dg category $C^\pretr$ (see Definition \ref{defwuph1}) with the weakly unital structure given in Proposition \ref{propph}. }
	\end{defn}

It is instructive to unwind the definition $\id_X=P_1(1_X)$ and get an explicit formula $\id_X$, $X\in C^\pretr$. 

Let $X=(\oplus X_i[r_i], q_{ij})\in C^\pretr$. We want to find the $(ij)$-component $(\id_X)_{ij}$. Let $i\le j$. Define
\begin{equation}\label{eqidx}
\begin{aligned}
\ &(\id_X)_{ij}=\sum_{i=\ell_0<\ell_1<\dots <\ell_k=j}\sum_{r=0}^k \\
&(-1)^r p_k(q_{\ell_{k-1}j}, q_{\ell_{k-2}\ell_{k-1}},\dots,q_{\ell_{r}\ell_{r+1}},1_{X_{i_r}},q_{\ell_{r-1}\ell_r},q_{\ell_{r-2}\ell_{r-1}},\dots,q_{i\ell_1})
\end{aligned}
\end{equation}
Then
\begin{equation}\label{eqidxx}
(\id_X)_{ij}=\begin{cases}
\text{the rhs of \eqref{eqidx}}&i\le j\\
0&i>j
\end{cases}
\end{equation}
The reader easily checks that for the case $\id_{\Cone(f)}$ \eqref{eqidx} gives \eqref{eqidcone}.
\comment
In the weakly unital dg-category $\mathcal{C}^{pretr}$, the weak units $id_X$ are $P_1(1_X)$, so that $id_X$ is closed by definition, once we've proven that $P$ is an $A_\infty$ map.\\
In particular, given a closed morphism $\xi: X \rightarrow Y \in \mathcal{C}$, we can describe $Id_{Cone(\xi)}$ as
\begin{equation*}
\begin{tikzpicture}[baseline= (a).base]
\node[scale=1.3] (a) at (5,5){
\begin{tikzcd}
X \arrow[d, "id_X"'] \arrow[r, "\xi"] \arrow[rd, "\epsilon"] & Y \arrow[d, "id_Y"]     \\
X \arrow[r, "\xi"]   & Y                                                  
\end{tikzcd}
};
\end{tikzpicture}
\end{equation*}
where $\epsilon = p_2(\xi,1) - p_2(1,\xi),$ since in our situation $\pa_{12} = \{(1,\xi), (\xi,1)\},$\\ $ \pa_{11} = \{(1_X)\}, \pa_{22} = \{(1_Y)\}$ and $\pa_{21} = \emptyset$.
So by our formula we get: 
\begin{equation}
\begin{aligned}
P_1(1_{Cone(\xi)})_{11} =&\;\; p_1(1_X) = id_X \\
P_1(1_{Cone(\xi)})_{12} =&\;\; p_2(\xi,1) - p_2(1,\xi) \\
P_1(1_{Cone(\xi)})_{22} =&\;\; p_1(1_Y) = id_Y \\
P_1(1_{Cone(\xi)})_{21} =&\;\; 0
\end{aligned}
\end{equation}
\endcomment

\section{\sc A Closed Model Structure on $\Cat_{dgwu}(\Bbbk)$}\label{s.mc}
\subsection{\sc }
In this section we provide a cofibrantly generated Quillen model structure on $\Cat_{dgwu}(\Bbbk)$.
The reader is referred to [Ho], [Hi], [GS] for general introduction to (cofibrantly generated) closed model categories. 
A simpler counterpart of the material of this Section is our treatment [PS] of the closed model structure on the category $\Cat_\dgwu^0(\k)$.
However, the passage from $\Cat_\dgwu^0(\k)$ to $\Cat_\dgwu(\k)$ requires several essentially new ideas, such as the weakly unital pretriangulated hull, introduced in Section \ref{s.def}, and a ``fattened version'' $\mathcal{K}$ of the Kontsevich dg category $\mathcal{K}$ (see Section \ref{sectionkcat}).

Define \emph{weak equivalences} $W$ in $\Cat_{dgwu}(\Bbbk)$ as the weakly unital dg functors $F: \C \rightarrow \D$ such that the following two conditions hold:
\begin{itemize}
	\item[(W1)] for any two objects $x,y \in \C$, the map of complexes $\C(x,y) \rightarrow \D(Fx,Fy)$ is a quasi-isomorphism of complexes,
	\item[(W2)] the functor $H^0(F): H^0(\C) \rightarrow H^0(\D)$ is an equivalence of $\Bbbk-$linear categories.
\end{itemize}
Remark that for a weakly unital dg category $\C$, the category $H^0(\C)$ is strictly unital and the functor $H^0(F)$ is well-defined.

Define \emph{fibrations} $Fib$ in $\Cat_{dgwu}(\Bbbk)$ as the weakly unital dg functors $F: \C \rightarrow \D$ such that the following two conditions hold:
\begin{itemize}
	\item[(F1)] for any two objects $x,y \in \C$, the map of complexes $\C(x,y) \rightarrow \D(Fx,Fy)$ is component-wise surjective,
	\item[(F2)] for any $x \in \C$ and a closed degree 0 arrow $g: Fx \mapsto z$ in $\D$ ($z$ not necessarily in the image of $F$), such that $g$ becomes an isomorphism in $H^0(\C),$ there is an object $y \in \C$ amd a closed degree 0 arrow $f: x \mapsto y$ inducing an isomorphism in $H^0(\D)$ and such that $F(f) = g$.
\end{itemize}
We define also a class $Surj$ of maps in $\Cat_\dgwu(\k)$ as follows: a weakly unital dg functor $F: \C \rightarrow \D$ belongs to $Surj$ if $F$ is surjective on objects and if $(F1)$ holds.
\begin{lemma}
	A weakly unital dg functor $F: \C \rightarrow \D$ belongs to $Fib \cap W$ if and only if it belongs to $Surj \cap (W1)$.
\end{lemma}
\begin{proof}
	A proof can be found in [PS, Lemma 2.1].
\end{proof}

\begin{lemma}\label{calculation}
Let $X$ be a weakly unital dg category, $x \in X$ an object. Suppose there are two degree $-1$ maps $h_1, h_2 \in X^{-1}(x,x)$ such that $dh_i = id_x, i = 1,2.$ Then there is $t \in X^{-2}(x,x)$ such that $dt = h_1 - h_2$.
\end{lemma}
\begin{proof}
Consider ${t} ^\prime= h_1h_2$. We find (using \eqref{eqp2}):
\begin{equation}
\begin{aligned}
 d{t}^\prime= & \; id_x \circ h_2 - h_1 \circ id_x \\
 = & \; -dp_2(1,h_2) + h_2 - p_2(1,id_x) + dp_2(h_1,1) - h_1 + p_2(id_x,1)
 \end{aligned}
 \end{equation}
If we manage to prove that $p_2(id,1) - p_2(1,id)$ is a boundary, we're done. Consider
\begin{equation}\label{eqh}
\begin{aligned}
H :=&\; p_2(p_2(id,1),1) + p_2(1,p_2(1,id)) - p_3(1,id,1) + p_3(id,id,1) + p_3(1,id,id) - p_3(1,1,id) \\
 & - p_3(id,1,1) + p_3(1,1,1)
\end{aligned}
\end{equation}
We compute $dH$ using \eqref{eqp2} and \eqref{eqp3}, the differential of each particular summand in \eqref{eqh} is displayed as $[\dots]$:
\begin{equation*}
\begin{aligned}
dH := \; &\big[\dotuline{p_{2}(id,1)\circ id} + \uline{p_{2}(id\circ id,1)}\big] + \big[\uuline{id\circ p_{2}(1,id)} - \uwave{p_{2}(1,id \circ id)}\big]+\big[\ \boxed{ p_{2}(id,1)  -p_{2}(1,id)} \\
&- \dashuline{id\circ p_{2}(id,1)} + \textcolor{orange}{p_{2}(1,id)\circ id}\big]+\big[ - \uline{p_{2}(id\circ id,1)} + \dashuline{id\circ p_{2}(id,1)} \big]+\\ 
&\big[ \uwave{p_{2}(1,id \circ id)} - \textcolor{orange}{p_{2}(1,id)\circ id}\big] + \big[{p_{2}(1,1)\circ id} - \uuline{id\circ p_{2}(1,id)}\big]+ \\ 
&\big[-\uline{\uwave{id\circ p_{2}(1,1)}} + \dotuline{p_{2}(id,1)\circ id}\big] + \big[\uline{\uwave{id\circ p_{2}(1,1)}} - {p_{2}(1,1)\circ id} \big] =\\
& p_2(id,1) - p_2(1,id).
\end{aligned}
\end{equation*}
Therefore $t := -{t}^\prime - p_2(1,h_2) + p_2(h_1,1) + H \in X^{-2}(x,x)$ is such that $dt = h_1 - h_2$.
\end{proof}

\begin{remark}{\rm
Lemma \ref{calculation} is a refined version of Lemma 2.8 and discussion below it of [PS]. When $\id_x\circ \id_x=\id_x$ is not assumed, the computation turns out to be more tricky, as we've seen in Lemma \ref{calculation}.	
}
\end{remark}

\subsection{\sc The weakly unital dg category $\mathcal{K}^\prime$}\label{sectionkcat}
The wu dg category $\mathcal{K}^\prime$ introduced below is a weakly unital counterpart of the dg category $\mathcal{K}$, due to Kontsevich in [K1, Lecture 6], and subsequently used by Tabuada in his closed model structure on $\Cat_\dg(\k)$ [Tab1]. 

Recall here the definition. 

The dg category $\mathcal{K}$ is the strictly unital dg category with two objects 0 and 1, and freely generated by $f\in \mathcal{K}^0(0,1), g\in\mathcal{K}^0(1,0), h_0\in \mathcal{K}^{-1}(0,0), h_1\in \mathcal{K}^{-1}(1,1), r\in \mathcal{K}^{-2}(0,1)$, whose differentials are
\begin{equation}\label{eqkstrict}
df=dg=0, dh_{0}=g\circ f-\id_{0}, dh_{1}=f\circ g-\id_{1}, dr=h_1\circ f-f\circ h_0
\end{equation}

Denote by $I_2$ the $\k$-linear envelope of the ordinary category with two objects $0$ and $1$, and with a unique morphism (including the identity one) between any ordered pair of objects.
There is a dg functor $p_\mathcal{K}\colon \mathcal{K}\to I_2$, which is the identity map on the objects, and sends $h_1, h_2,r$ to 0.

The following well-known result says that $\mathcal{K}$ is a semi-free resolution of $I_2$:
\begin{lemma}\label{lemmadrk}
The dg functor $p_\mathcal{K}\colon \mathcal{K}\to I_2$ is a quasi-equivalence.	
\end{lemma}
The proof can be found in [Dr, 3.7].

\qed

\begin{defn}
	{\rm
Denote by $\K$ the weakly unital dg category with two objects $0$ and $1$, whose morphisms are freely generated by the following morphisms:
\begin{itemize}
\item a  morphism $f \in (\K)^0(0,1)$,
\item a  morphism $g \in (\K)^0(1,0)$,
\item a morphism $h_0 \in (\K)^{-1}(0,0)$,
 \item a  morphism  $h_1 \in (\K)^{-1}(1,1)$,
 \item a degree -2 morphism $r\in (\mathcal{K}^{\prime})^{-2}(0,1)$
 \end{itemize}
 whose differentials are given as
\begin{equation}\label{eqrelk}
\begin{aligned}
\ &df=dg=0\\
&dh_0=g\circ f -\id_0, \;\;\; dh_1=f\circ g-\id_{1}\\
&dr = h_1\circ f - f\circ h_{0} + p_2(1,f) - p_2(f,1)
\end{aligned}
\end{equation} \\
}
\end{defn}

A version of a lemma in [K1, Lecture 6] holds as well in the setting of weakly unital dg categories: 
\begin{lemma}\label{lem.kon}
Let $C$ be a weakly unital dg category, and $\xi \in C^{0}(x,y)$  be a closed degree 0 morphism, such that $[\xi] \in H^0(C)$ is a homotopy equivalence. Then there is a weakly unital dg functor $F: \K \rightarrow C$, such that $F(f) = \xi.$
\end{lemma}
\begin{proof}
	By definition of being a homotopy equivalence, there exist $\eta \in C^0(y,x)$, $h_x \in C^{-1}(x,x)$ and $h_y \in C^{-1}(y,y)$ such that:
	\begin{equation}
	\begin{aligned}
	dh_x = \eta\circ\xi - id_x \\
	dh_y = \xi\circ\eta - id_y
	\end{aligned}
	\end{equation}
	Now we are looking for a morphism $r \in C^{-2}(x,y)$ such that $dr = h_y\circ\xi - \xi\circ h_{x} + p_2(1,\xi) - p_2(\xi,1)$. 
	We define $A := h_y\circ\xi - \xi\circ h_{x} + p_2(1,\xi) - p_2(\xi,1)$. Clearly $dA = 0$. \\
	Then we take $h^{'}_y := h_y - A\circ\eta$ and $r:= A\circ h_x - p_2(A,1)$, so that we easily get:
	\begin{equation}
	dh^{'}_y = dh_y - dA\eta + Ad\eta =\xi\eta - id_y. 
	\end{equation}
	and also
	\begin{equation}
	\begin{aligned}
	dr =& dA\circ h_x - A\circ dh_x - A\circ id_x + A - p_2(dA,1) \\
	=& -A\circ(\eta\xi - id_x) - A\circ id_x + A \\
	=& -A\circ \eta\circ \xi + h_y\circ\xi - \xi\circ h_x + p_2(1,\xi) - p_2(\xi,1) \\
	=& (h_y - A\circ\eta)\circ\xi - \xi\circ h_x  + p_2(1,\xi) - p_2(\xi,1) \\
	=& h^{'}_y\circ\xi - \xi\circ h_x  + p_2(1,\xi) - p_2(\xi,1).
	\end{aligned}
	\end{equation}
	We are done.
\end{proof}

We prove a lemma which we will be used later in the proof of Theorem \ref{firstbigthm} and (implicitly) in Theorem \ref{bigtheorem}:
\begin{lemma}\label{kcontr}
Let $C$ be a weakly unital dg category. There is a bijection between the set of weakly unital dg functors from $\K$ to $C$ and the set of pairs $(\xi,h)$, where $\xi \in C^{0}(x,y)$ is a closed morphism and $h$ is a contraction of $Cone(\xi)$ in $C^{pretr}$.
\end{lemma}
\begin{proof}
A wu dg functor $F\colon \K\to C$ amounts to the following morphisms in $C$: $\xi=F(f), \eta=F(g), h_{11}=F(h_0), h_{22}=F(h_1), h_{12}=F(r)$ such that
\begin{equation}\label{eq.k}
d(\xi)=0, \;\;  d\eta = 0,\;\; dh_{11} = \eta\circ\xi - id_x, \;\; dh_{22} = \xi\circ\eta - id_y, \;\; dh_{12} = h_{22}\circ\xi -\xi\circ h_{11} + p_2(1,\xi) - p_2(\xi,1)
\end{equation}
A contraction to $id_{Cone(\xi)}$ (see \eqref{eqidcone}) is the datum of a morphism $H: Cone(\xi) \rightarrow Cone(\xi)$ of degree -1 such that 
\begin{equation}\label{eq.contr}
dH = Id_{Cone(\xi)},
\end{equation}
Then $$H=(h_{11}\in C^{-1}(X,X),h_{22}\in C^{-1}(Y,Y), h_{12}\in C^{-2}(X,Y),h_{21}\in C^0(Y,X))$$
as in the diagram below
\comment
\begin{equation*}
\begin{tikzpicture}[baseline= (a).base]
\node[scale=1.3] (a) at (5,5){
\begin{tikzcd}
X \arrow[d, "h_{11}"'] \arrow[r, "\xi"] \arrow[rd, dotted, "h_{21}"] & Y \arrow[d, "h_{22}"] \arrow[dl, dotted, "h_{12}"]     \\
X \arrow[r, "\xi"]   & Y                                                  
\end{tikzcd}
};
\end{tikzpicture}
\end{equation*}
\endcomment
\begin{equation}
\xymatrix{
X\ar[rr]^{\xi}\ar[d]_{h_{11}}\ar[drr]^(0.35){{h_{12}}}&&Y\ar[d]^{h_{22}}\ar[dll]^(0.6){h_{21}}\\
X\ar[rr]_{{\xi}}&&Y
}
\end{equation}
We see by a direct calculation  that  \ref{eq.k} is equivalent to \ref{eq.contr}. 
\end{proof}

\comment
There is the following alternative construction of the wu dg category $\mathcal{K}^\prime$, which we are going to use later in the proof of \ref{firstbigthm}.

Let $\Gamma$ be the dg graph with two vertices 0 and 1 such that
\begin{equation}\label{eqgamma}
\Gamma(0,1)=\Gamma(0,0)=\Gamma(1,1)=\k,\ \  \Gamma(1,0)=0
\end{equation} 
Let $\mathcal{A}_1$ be the free wu dg category generated by $\Gamma$. Denote by $f\in\mathcal{A}_1(0,1)$ the generator. Consider $\mathcal{A}_1^\pretr$, it contains an object $\Cone(f)$. Add freely a new morphism $h\in \mathcal{A}_1^\pretr(\Cone(f),\Cone(f))$ of degree -1 such that $dh=\id_{\Cone(f)}$. Denote the resulting wu dg category by $\mathcal{A}_2$. Denote finally by $\mathcal{A}_3$ the full subcategory of $\mathcal{A}_2$ having two objects 0 and 1, clearly it is a weakly unital dg category. 

\begin{coroll}\label{corka3}
The wu dg category $\mathcal{K}^\prime$ is isomorphic to the wu dg category $\mathcal{A}_3$.
\end{coroll}
\begin{proof}
	One easily describes the set $\Cat_\dgwu(\mathcal{A}_3,C)$ for a wu dg category $C$, because $\mathcal{A}_3$ is a free wu dg category. Thus, to define a dg functor $F\colon \mathcal{A}_3\to C$ ammounts to the data of a morphism $\xi\in C^0(x,y)$ for some objects $x,y\in C$, and a morphism $H\in C^\pretr(\Cone(f),\Cone(f))$ of degree -1 such that $dH=\id_{\Cone(f)}$. Recall that $\id_{\Cone(f)}$ is given by \eqref{eqidcone}, cf. \eqref{eq.contr}. Unwinding the definition, we get the same equations as in \eqref{eq.k}.
	
	Now the statement follows from the Yoneda lemma. 
\end{proof}
\endcomment

\begin{prop}\label{firstbigthm}
	The wu dg category $\K$ has the same homotopy type as the Kontsevich dg category $\mathcal{K}$. More precisely, regarding $\mathcal{K}$ as an object in $\Cat_\dgwu(\k)$, the natural projection $p\colon \K\to\mathcal{K}$, sending all $p_n(-)$, $n\ge 2$, to 0, is a weak equivalence. 
\end{prop}
\comment

\begin{proof}

Let $A$ be the unital dg category with a single object $0$, and with morphisms $A(0,0)=\k \id$.  Likewise, denote by $A^\prime$ the weakly unital dg category, generated by the graph with a single vertex and a single (identity) arrow. 

Consider the commutative diagram:
\begin{equation}
\xymatrix{
\mathcal{A}\ar[d]_{p_1}\ar[r]^{\kappa}&\mathcal{K}^\prime\ar[d]_{p_2}\ar[r]&\mathcal{I}_2\ar[d]_{p_3}\\
A\ar[r]^{\sim}&\mathcal{K}\ar[r]^{\sim}&I_2
}
\end{equation} 

        \begin{equation*}
\begin{tikzpicture}[baseline= (a).base]
\node[scale=1.3] (a) at (5,5){
        \begin{tikzcd}
                      A \arrow[d, "\sim"] \arrow[r, "\kappa"]      & \K \arrow[d] \arrow[r] & \mathcal{I}_2 \arrow[d, "\sim"]  \\
A \arrow[r, "\sim"] & \mathcal{K} \arrow[r, "\sim"]            & I_2
\end{tikzcd}
};
\end{tikzpicture}
\end{equation*}

Since $A$ and $\mathcal{A}$ are freely generated as strictly unital (resp., weakly unital) dg categories by the same dg graph, the quasi-isomorphism of dg-operad $\mathcal{O}'$ and $\mathcal{A}ssoc_+$ ??? implies that $\mathcal{A}$ is quasi- equivalent to $A$. The same reasoning holds for $\mathcal{I}_2$ and $I_2$. That is, $p_1$ and $p_3$ are quasi-equivalences. 
The bottom row arrows are quasi-equivalences by ???. \\

The statement is that $p_2$ is a quasi-equivalence. 
By 2-of-3, it is enough to prove that $\kappa: \mathcal{A} \rightarrow \K$ is a quasi-equivalence.

In order to describe the hom-complexes of $\K$ we use the description of $\K$ given in Corollary \ref{corka3}, as the full wu dg subcategory of $\mathcal{A}_2$. 
Let $a,b\in \{0,1\}$, consider an ascending filtration $\{F_n(a,b)\}_{n\ge 0}$ of $\K(a,b)$:
	\begin{center} 
	$F_n(a,b)$ := \{morphisms from $a$ to $b$ containing $\le n$ factors $h$\}
	\end{center}
We have $F_0(a,b) = \mathcal{A}_1^\pretr(a,b)$. 
Consider the spectral sequence on the complex $\K(0,a)$, $a\in\{0,1\}$, associated with the filtration $F_\ldot(0,a)$. 
Associate the $x$-axis with the filtration, and the $y$-axis with the cohomological degree. Then the spectral sequence lives in the III quarter, and therefore converges.

Then the result follows from Lemmas \eqref{lem.tab} and \eqref{lemmafil1} below. 
	\begin{lemma}\label{lem.tab}
		For any $a\in \{0,1\}$
	one has: $\mathcal{A}^\pretr_1(0,a)\sim I_2(0,a)=\k$ (where $\sim$ stands for a quasi-isomorphism). Moreover, $\mathcal{A}^\pretr_1(0,\Cone(f))$ is acyclic. 
	\end{lemma}
	\begin{proof}[Proof of Lemma \ref{lem.tab}]
	We remark that the canonical inclusion $\C \hookrightarrow \C^{pretr}$ is fully faithful. Therefore 
	\begin{equation}\label{eq.std}
\mathcal{A}^\pretr_1(a,b) \sim \mathcal{A}_1(a,b)
	\end{equation}
	Since the operads $\mathcal{O}'$ and $A ssoc_+$ are quasi-isomorphic, it follows that the complexes $\mathcal{A}_1(a,b)$ are quasi-isomorphic to the corresponding complexes $A_1(a,b)$ in the strictly unital case. (Where for the strictly unital case $A_1$ is the unital dg category generated by the dg graph $\Gamma$, see \eqref{eqgamma}). It gives the first statement.\\
	For the second statement, remark that for any closed morphism $f\colon x\to y$ in a pretriangulated dg category $C$, for any object $a$ in $C$  there is a distignuished triangle in the category of complexes:
	\begin{equation}
C(a,x)\xrightarrow{f_*}C(a,y)\to C(a,\Cone(f))\xrightarrow{+1}
	\end{equation}
	It gives two distignuished triangles in the category of complexes and a map of distignuished triangles:
	\begin{equation}
	\xymatrix{
	\mathcal{A}^\pretr_1(0,0)\ar[r]\ar[d]_{t_{1*}}&\mathcal{A}^\pretr_1(0,1)\ar[r]\ar[d]_{t_{2*}}&\mathcal{A}^\pretr_1(0,\Cone(f))\ar[r]^{\hspace{12mm}+1}\ar[d]_{t_{3*}}&\\
	A^\pretr_1(0,0)\ar[r]&A^\pretr_1(0,1)\ar[r]&A^\pretr_1(0,\Cone(f))\ar[r]^{\hspace{12mm}+1}&
	}
	\end{equation}
We pass to the long exact sequences in cohomology and use the proven statements that $t_{1*}$ and $t_{2*}$ are quasi-isomorphisms. It follows from the 5-lemma that $t_{3*}$ is a quasi-isomorphism as well.

	Now we can exploit the pretriangulated structure: clearly \[ x \rightarrow y \rightarrow Cone(f) \rightarrow x[1] \]
	is a distinguished triangle, and since $\forall a \in \{x,y,Cone(f)\}, A_1(a, -)$ and $A(a, -)$ are homological functors, we get the long exact sequences:
	\begin{equation}\label{first.les}
	.. \rightarrow A_1(a,x) \rightarrow A_1(a,y) \rightarrow A_1(a,Cone(f)) \rightarrow .. 
	\end{equation}
	and 
	\begin{equation}\label{second.les}
	.. \rightarrow A_1(a,x) \rightarrow A_1(a,y) \rightarrow A_1(a,Cone(f)) \rightarrow .. 
	\end{equation}
	We can consider the projection from \ref{first.les} to \ref{second.les} which sends all the $p_n(..)$ to $0$.
	By the remark above, when $a \in \{x,y\}$ we have got quasi isomorphisms $A_1(a,x) \cong A_1(a,x)$ and $A_1(a,y) \cong A_1(a,y)$:
	\begin{equation*}
\begin{tikzpicture}[baseline= (a).base]
\node[scale=1] (a) at (5,5){
	\begin{tikzcd}
.. \arrow[r] & A_1(a,x) \arrow[d, "\sim"] \arrow[r] & A_1(a,y) \arrow[d, "\sim"] \arrow[r] & A_1(a,Cone(f)) \arrow[d] \arrow[r] & .. \\
.. \arrow[r] &  A_1(a,x) \arrow[r]           &  A_1(a,y) \arrow[r]           &  A_1(a,Cone(f)) \arrow[r]           & ..
\end{tikzcd}
};
\end{tikzpicture}
\end{equation*}
        which gives us the quasi isomorphism $A_1(a,Cone(f)) \cong A_1(a,Cone(f))$.
        Similarly we have that $A_1(Cone(f),a) \cong A_1(Cone(f),a), \forall a \in \{x,y\}$. \\
        Considering this other map:
        \begin{equation*}
\begin{tikzpicture}[baseline= (a).base]
\node[scale=1] (a) at (5,5){
	\begin{tikzcd}
.. \arrow[r] & A_1(Cone(f),Cone(f)) \arrow[d] \arrow[r] & A_1(y,Cone(f)) \arrow[d, "\sim"] \arrow[r] & A_1(x,Cone(f)) \arrow[d, "\sim"] \arrow[r] & .. \\
.. \arrow[r] &  A_1(Cone(f),Cone(f)) \arrow[r]           &  A_1(y,Cone(f)) \arrow[r]           &  A_1(x,Cone(f)) \arrow[r]           & ..
\end{tikzcd}
};
\end{tikzpicture}
\end{equation*}
we get the last desired quasi isomorphism: $A_1(Cone(f),Cone(f)) \cong A_1(Cone(f),Cone(f))$.

	\end{proof}

\begin{lemma}\label{lemmafil1}
The quotient-complexes $F_n(0,a)/F_{n-1}(0,a)$, $a\in \{0,1\}$, $n\ge 1$, are acyclic.	
\end{lemma}
\begin{proof}
For a given $n\ge 1$, consider an ascending filtration $\{\overline{\phi}_{\ell,n}\}_{\ell\ge 0}$, defined as follows. For a ``monomial'' element $\gamma$ in $F_n(a,b)$ set 
\begin{equation}
S(\gamma)=\sum_{\substack{{\text{all occurences of}}{\text{ $p_s(...,h,...)$ in $\gamma$}}}}
(s-1)
\end{equation}	
where the sum is taken over all $p_s(...)$ in $\gamma$, having $h$ as one of its arguments. Define
\begin{equation}
\phi_\ell F_n(a,b)=\{\text{the linear combination of all elements  $\gamma\in F_n(a,b)$ with $S(\gamma)\le\ell$} \}
\end{equation}
\begin{equation}
\overline{\phi}_{\ell,n}(a,b)=\frac{\phi_\ell F_n(a,b)}{F_{n-1}(a,b)\cap\phi_\ell F_n(a,b)}
\end{equation}
\\	
A key observation is that the ``operadic'' component of the differential {\it strictly lower} $\ell=\sum(s-1)$. The only component which preserves it comes from the ``inner" differential in $A_1(a,b).$ \\
It follows that $\overline{\phi}_{\ell,n}$ is the {\it tensor product of complexes}
\begin{equation}
\overline{\phi}_{\ell,n}(0,a)= \mathcal{A}^\pretr_1(Cone(f),a) \otimes h \otimes  \mathcal{A}^\pretr_1(Cone(f),Cone(f)) \otimes \dots \otimes h \otimes \mathcal{A}^\pretr_1(0,Cone(f))
\end{equation}
with exactly $n\ge 1$ factors $h$ and having the number $S$ equal to $\ell$, and the differential is the ``inner'' differential in $\mathcal{A}_1^\pretr$, not acting on $h$.  

This is indeed the case, since the differential in $\phi_\ell / \phi_{\ell - 1}$ on a general element \\ $p_s(g_1, .., h, .., g_s)\cdot .. \cdot p_m(f_1, .., h, .., f_m),$ where $g_j, f_i$ belongs to an appropriate $A_1(a,b)$ acts like this:
\[ \sum_j p_s(g_1, .., dg_j, .., h, .., g_s) \cdot \; .. \; \cdot p_m(f_1, .., h, .., f_m) + \;...\; + \sum_i p_s(g_1, .., h, .., g_s)\cdot \;..\; \cdot p_m(f_1, .., df_i, .., h, .., f_m), \]
where the operadic terms $p_s$ are not relevant, so that $\phi_\ell / \phi_{\ell - 1}$ are equivalent (as complexes) to:
\[ A_1(Cone(f),y) \otimes h \otimes  A_1(Cone(f),Cone(f)) \otimes .. \otimes h \otimes A_1(x,Cone(f)). \]

By Lemma \ref{lem.tab} the complex $\mathcal{A}^\pretr_1(0,Cone(f))$ is acyclic, so 
$\overline{\phi}_{\ell,n}$ is acyclic by the K\"{u}nneth formula. 
\\
The spectral sequence associated with the filtration $\overline{\phi}_{\ell,n}$ on $F_n(0,a)/F_{n-1}(0,a)$ clearly converges by dimensional reasons. 
It follows that, for $n \geq 1$, the complex $F_n / F_{n-1}$ is acyclic.

\end{proof}

\end{proof}

\endcomment

\begin{proof}
	Consider ascending filtrations $\{\Phi_i(a,b)\}_{i\ge 0}$ of $\K(a,b)$ and $\{F_i(a,b)\}_{i\ge 0}$ of $\mathcal{K}(a,b)$, $a,b\in\{0,1\}$, such that
	$p(\Phi_i(a,b))\subset F_i(a,b)$, $i\ge 0$. We prove that the corresponding spectral sequences converge, and that the map $p$ induces an isomorphism in the $E_1$ sheets. The result will follow from the latter statement. 
	
	Define $F_i(a,b)$ as the dg vector space generated by all monomials with $\le i$ factors $r$. Define $\Phi_i(a,b)$ similarly, but we count all occurances of $r$ in expressions $p_j(...,r,...)$ as a ``factor $r$''. 
	It is clear that $d(F_i(a,b))\subset F_i(a,b)$ and $d(\Phi_i(a,b))\subset\Phi_i(a,b)$, and 
	that $p(\Phi_i(a,b))\subset F_i(a,b)$. 
	
	Also, it is clear that both spectral sequences converge, by dimensional reasons (the spectral sequences live in the quarter ``$x\le 0, y\le 0$''). 
	
	We have:
	\begin{lemma}\label{lemmae1}
	The map $p$ induces an isomorphism in the $E_1$ sheets.
	\end{lemma}
	\begin{proof}
	For both cases, the differential in $E_0$ is the same as it would be if $dr=0$. Therefore, to compute $E_1$ we assume that $dr=0$ for both cases.
	
	Denote by $\tilde{\K}$ (corresp. $\tilde{\mathcal{K}}$) the semi-free wu dg category (corresp., the semi-free unital dg category) with two objects $\{0,1\}$, the generators $f,g,h_0,h_1,r$, as in \eqref{eqkstrict}, \eqref{eqrelk}, and in which the differential of the generators is given by {\it the same} formulas:
	\begin{equation}\label{eqk0}
	\begin{aligned}
	\ &df=dg=0\\
	&dh_0=g\circ f -\id_0, \;\;\; dh_1=f\circ g-\id_{1}\\
	&dr = 0
	\end{aligned}
	\end{equation}
	Now the statement follows from the fact that the projection of dg operads $\mathcal{O}^\prime\to \mathcal{O}$ is a quasi-isomorphism in any airity, see [PS, Section 4] (here $\mathcal{O}$ is a dg operad, defined analogously to $\mathcal{O}^\prime$, see loc.cit.). 
	\end{proof}
	
	\end{proof}

\begin{remark}{\rm
The argument employed in the proof of Lemma \eqref{lemmae1}
can not be used directly for $p\colon \K\to \mathcal{K}$ (without any spectral sequence argument), because $dr$ is given by {\it different} formulas in \eqref{eqkstrict} and \eqref{eqrelk}. More precisely,  the equation for $dr$ 	for $\K$ is a {\it deformation} of that for $\mathcal{K}$. Consequently, it does not follow directly from the quasi-isomorphism $\mathcal{O}^\prime\to\mathcal{O}$ that $p\colon \K(a,b)\to\mathcal{K}(a,b)$ is a quasi-isomorphism.
}
\end{remark}

\comment
We will use the following alternative description of $\K$.

Let $\Gamma$ be the dg graph with two vertices 0 and 1 such that
\begin{equation}\label{eqgamma}
\Gamma(0,1)=\Gamma(0,0)=\Gamma(1,1)=\k,\ \  \Gamma(1,0)=0
\end{equation} 
Let $\mathcal{A}_1$ be the free wu dg category generated by $\Gamma$. Denote by $f\in\mathcal{A}_1(0,1)$ the generator. Consider $\mathcal{A}_1^\pretr$, it contains an object $\Cone(f)$. Add freely a new morphism $h\in \mathcal{A}_1^\pretr(\Cone(f),\Cone(f))$ of degree -1 such that $dh=\id_{\Cone(f)}$. Denote the resulting wu dg category by $\mathcal{A}_2$. Denote finally by $\mathcal{K}^\pprime$ the full subcategory of $\mathcal{A}_2$ having two objects 0 and 1, clearly it is a weakly unital dg category. 

\begin{prop}
	There is a dg functor $p\colon \K\to\mathcal{K}^\pprime$ which is an isomorphism.
\end{prop}
\endcomment

Recall that $I_2$ is the $\k$-linear envelope of the ordinary category with two objects, 0 and 1, and having exactly 1 morphism between any ordered pair of objects. Recall that Lemma \ref{lemmadrk} says that the projection $p_\mathcal{K}\colon \mathcal{K}\to I_2$ is a weak equivalence. 
\begin{coroll}\label{corki2}
The natural projection $\K\to I_2$, equal to the composition $\K\xrightarrow{p}\mathcal{K}\xrightarrow{p_\mathcal{K}} I_2$, is a weak equivalence. 
\end{coroll}
\begin{proof}
	The dg functor $p$ is a weak equivalence by Proposition
	\ref{firstbigthm}. The dg functor $p_\mathcal{K}\colon \mathcal{K}\to I_2$, is a weak equivalence by Lemma \ref{lemmadrk}. 
\end{proof}

\subsection{\sc The sets $I$ and $J$} 
Denote by  $D(n)$ the complex $0\to\Bbbk [n] \xrightarrow{\id} \Bbbk [n-1]\to 0$, it is  $D(n)=Cone(id: \Bbbk [n] \rightarrow \Bbbk [n])$. \\
Denote $S(n-1)= \Bbbk [n-1]$.
Consider the natural imbedding $i: S(n-1) \rightarrow D(n)$  of complexes. \\
Denote by $A$ the weakly unital dg category with a single object $0$ and generated (over the dg operad $\mathcal{O}^\prime$) by  $id_0$. Denote by $\kappa$ the weakly unital dg functor \[\kappa: A\rightarrow \K,\] sending $0$ to $0$. It follows from Corollary \ref{corki2} that $\kappa$ is a weak equivalence. \\
Denote by $\mathcal{B}$ the weakly unital dg category with two objects $0$ and $1$ and generated over $\mathcal{O}^\prime$ by morphisms $id_0$ and $id_1$. \\
Let ${P}(n)$ be the dg graph with two objects $0$ and $1$, and with morphisms ${P}(n)(0,1) = D(n), {P}(n)(0,0) = 0, {P}(n)(1,1) = 0, {P}(n)(1,0) = 0$. Denote by $\mathcal{P}(n)$ the weakly unital dg category generated by $P(n)$: $\mathcal{P}(n)= FU(P(n))$. \\
Denote by $\alpha(n)$ the weakly unital dg functor \[\alpha(n): \mathcal{B} \rightarrow \mathcal{P}(n),\] sending $0$ to $0$ and $1$ to $1$. \\
Let $C(n)$ be the dg graph with two objects $0$ and $1$, and with morphisms $C(n)(0,1) = S(n-1), C(n)(0,0) = 0, C(n)(1,1) = 0, C(n)(1,0) = 0$. Denote by $\mathcal{C}(n)$ the weakly unital dg category generated by $C(n)$: $\mathcal{C}(n) := FU(C(n))$. \\
Let $b(n): C(n) \rightarrow P(n)$ be a map of dg graphs sending $0$ to $0$, $1$ to $1$, and such that $S(n-1) = C(n)(0,1) \xrightarrow{i} P(n)(0,1) = D(n)$ is the imbedding $i$. Denote by $\beta(n)$ the weakly unital dg functor \[\beta(n) := FU(b(n)) : \mathcal{C}(n) \xrightarrow{i} \mathcal{P}(n).\]
Denote by $Q$ the natural weakly unital dg functor \[Q: \emptyset \rightarrow \mathcal{A}.\]

Let $I$ be a set of morphisms in $\Cat_{\dgwu}(\Bbbk)$ which comprises the weakly unital dg functors $Q$ and $\beta(n), n \in \mathbb{Z}$.

Let $J$ be a set of morphisms in $\Cat_{\dgwu}(\Bbbk)$ which comprises $\kappa$ and $\alpha(n), n \in \mathbb{Z}$. 

The set $I$ and $J$ are referred to as the sets of \emph{generating cofibrations} and of \emph{generating acyclic cofibrations}, correspondingly. \\

\begin{lemma}\label{lemma.simple}
	A weakly unital dg functor $\phi: \C \rightarrow \D$ has $RLP$ with respect to all $\alpha(n), n \in \mathbb{Z}$ if and only if $\phi$ obeys $(F1)$. A weakly unital dg functor $\phi: \C \rightarrow \D$ has $RLP$ with respect to all $\beta(n), n \in \mathbb{Z}$ if and only if $\phi$ obeys $(F1) \cap (W1)$.
\end{lemma}
\begin{proof}
	The proof is straightforward and can be found in [PS, Sec.2.1.3].
\end{proof}

Recall the standard terminology (conventional for the theory of closed model categories): a dg functor $\phi\colon \C\to \D$ belongs to $I\text{-}inj$ (resp., to $J\text{-}inj$) if it has the RLP with respect to all morphisms in $I$ (resp., in $J$). 

\begin{prop}\label{propt5}
	One has 
	\[ I\text{-}inj = Surj \cap (W1) = J\text{-}inj \cap W \]
\end{prop}
\begin{proof}
	Thanks to Lemma \ref{lemma.simple}, the first equality follows from the fact that a dg functor $\phi\colon \C\to\D$ has the RLP for $Q$ if and only if it is surjective on objects, which is straightforward. \\
	The second equality is far more sophisticated, and its proof is based on the following lemma.\footnote{Lemma \ref{Fibinj} is one of the most subtle places in our constructions; in particular, the theory of weakly unital pre-triangulated hull developed in Section \ref{sectionptrh}, and Lemma \ref{calculation}, were designed especially for its proof.}
\begin{lemma}\label{Fibinj}
One has $Fib = J\text{-}inj$.
\end{lemma}
\begin{proof}
The inclusion $J\text{-}inj \subseteq Fib$ follows from Lemma \ref{lem.kon}. \\
In order to prove the inclusion $Fib \subseteq J\text{-}inj$  consider $\phi: \C \rightarrow \D$ in $Fib$. Axiom $(F1)$ is equivalent to the $RLP$ with respect to $\alpha(n), n \in \mathbb{Z}$, thence we only need to prove the $RLP$ with respect to $\kappa$ for $\phi$.
We are given a weakly unital dg functor $F: \K \rightarrow \D$. We can apply $(F2)$ to $\xi = F(f) \in \D^0(\phi(x),z)$, and so we get a morphism $\eta \in \C^0(x,y)$ which is a homotopy equivalence and $\phi(\eta) = \xi, \phi(y) = z$. (Recall that $f,g,h_0,h_1,r$ are generators for $\K$, see \eqref{eqrelk}).\
We should construct a weakly unital dg functor $\hat{F}: \K \rightarrow \C$ such that $\phi \circ \hat{F} = F$ and $\hat{F}(f) = \eta$. 
By Lemma \ref{kcontr} having a weakly unital dg functor $F: \K \rightarrow \D, F(f) = \xi$ is equivalent to having a contraction of $Cone(\xi)$ in $\D^{pretr}$, i.e. we have $h \in \D^{pretr}(Cone(\xi),Cone(\xi))$ such that $dh = id_{Cone(\xi)}$.
By $(F2)$ we know that $Cone(\eta)$ is also contractible, so we also have a morphism $\tilde{h}_1 \in \C^{pretr}(Cone(\eta),Cone(\eta))$, such that $d\tilde{h}_1 = id_{Cone(\eta)}$. Even though we don't know whether $\phi^{pretr}(\tilde{h}_1) = h$, we still have $d\phi^{pretr}(\tilde{h}_1) = id_{Cone(\xi)}$. \\
By Lemma \ref{calculation}, one has $\phi^{pretr}(\tilde{h}_1) - h = dt$. By $(F1)$, we find a lift $\tilde{t}$ of $t$, and set $\tilde{h} := \tilde{h}_1 - d\tilde{t}$. One clearly has $d\tilde{h} = id_{Cone(\eta)}$ and $\phi(\tilde{h}) = h$. This gives the desired lift $\hat{F}: \K \rightarrow \C$ such that $\phi \circ \hat{F} = F$, by Lemma \ref{kcontr}. 
\end{proof}
Now we have $J\text{-}inj \cap W = Fib \cap W = Surj \cap (W1)$, and we're done.
\end{proof}

The following theorem is one of our main results:
\begin{theorem}\label{bigtheorem}
	The category $\Cat_{\dgwu}(\Bbbk)$ admits a cofibrantly generated closed model structure whose weak equivalences and fibrations are as above, and whose sets of generating cofibrations and generating acyclic cofibrations are $I$ and $J$. 
\end{theorem}

\subsection{\sc Proof of Theorem \ref{bigtheorem}}
We refer the reader to standard textbooks on closed model categories for the notations $I\text{-}cell$, $J\text{-}cell, I\text{-}cof, J\text{-}cof$, see e.g. [Ho, 2.1.3], [Hi, 11.1]. Recall that $I\text{-}cell\subset I\text{-}cof$ and $J\text{-}cell\subset J\text{-}cof$. 

Recall [Ho, Th. 2.1.19] which the proof is based on:
\begin{theorem}
	Let $C$ be a small complete and cocomplete category. Suppose that $\mathcal{W}$ is a subcategory of $C$, and $I$ and $J$ are sets of maps. Assume that the following conditions hold:
	\begin{itemize}
		\item[1.] the subcategory $W$ has 2-out of-3 property and is closed under retracts,
		\item[2.] the domains of $I$ are small relative to $I$-cell,
		\item[3.] the domains of $J$ are small relative to $J$-cell,
		\item[4.] $J\text{-}cell \subset W \cap I\text{-}cof$,
		\item[5.] $I\text{-}inj = W \cap J\text{-}inj$.
	\end{itemize}
    Then there is a cofibrantly generated closed model structure on $C$, for which the morphisms $W$ of $\mathcal{W}$ are weak equivalences, $I$ are generating cofibrations, $J$ are acyclic generating cofibrations. Its fibrations are defined as $J\text{-}inj.$
\end{theorem}

\begin{proof}[Proof of Theorem \ref{bigtheorem}]
	The category $\Cat_\dgwu(\k)$ is small complete and small cocomplete by [PS, Theorem 1.19]. 
	The conditions $(1)-(3)$ are clear. Condition $(5)$ was proved in Proposition \ref{propt5}. It follows from $(5)$ that $I\text{-}inj \subset J\text{-}inj$, and so $J\text{-}cof \subset I\text{-}cof$.
	It only remains to prove that $J\text{-}cell \subseteq W$.
\begin{proof}[Proof of $J\text{-}cell \subseteq W$]\label{proofJcell}
	We have to prove that the push-out of a morphism in $J$ is a weak equivalence. We consider two cases: when the morphism in $J$ is $\alpha(n), n\in\mathbb{Z}$, and when it is $\kappa$.  \\
	\emph{First case.} \\
	The first case we need to consider is shown in the push-out diagram: 
	\begin{equation}
	\begin{tikzpicture}[baseline= (a).base]
\node[scale=1] (a) at (5,5){
	\begin{tikzcd}
 \mathcal{B} \arrow[d, "\alpha(n)"'] \arrow[r, "g"] & \X \arrow[d, "f"] \\
 \mathcal{P}(n) \arrow[r]           &  \Y
\end{tikzcd}
};
\end{tikzpicture}
	\end{equation}
	where $g$ is an arbitrary map. 
	We need to prove that $f$ is a weak equivalence. \\
	Clearly $Ob(\X) = Ob(\Y)$, and $f$ acts by the identity maps on the objects. 
	We are left to show that, for any objects $a,b \in Ob(\X)$, the map of complexes $f(a,b): \X(a,b) \rightarrow \Y(a,b)$ is a quasi-isomorphism. 
	For objects $0,1$ in $Ob(\mathcal{B})$, let denote by $u = g(0)$ and $v=g(1)$.
By Proposition \ref{propappc}, one has the following description for the hom-complexes of $\Y$:
	\begin{equation}
	\begin{aligned}
	\Y(a,b) := \;\X(a,b) &\bigoplus \mathcal{O}'(3) \otimes \X(v,b)\otimes D(n) \otimes \X(a,u) \\
	&\bigoplus \mathcal{O}'(5) \otimes \X(v,b)\otimes D(n) \otimes \X(v,u)\otimes D(n) \otimes \X(a,u) \bigoplus \cdot\cdot\cdot 
	\end{aligned}
	\end{equation}
	The map $f(a,b)$ sends $\X(a,b)$ to the first summand. The other summands have trivial cohomology by the K\"{u}nneth formula, since the acyclicity of $D(n)$.\\
	\vspace{1mm}
	\\
\emph{Second case.} \\
	As the second case we consider the following push-out diagram:
	\begin{equation}
	\begin{tikzpicture}[baseline= (a).base]
\node[scale=1] (a) at (5,5){
	\begin{tikzcd}
 A \arrow[d, "\kappa"'] \arrow[r, "h"] & \X \arrow[d, "f"] \\
 \K \arrow[r]           &  \Y
\end{tikzcd}
};
\end{tikzpicture}
	\end{equation}
	where $h$ is an arbitrary map.

	One has $Ob(\Y) = Ob(\X) \sqcup 1_{\K}$. It is clear that $H^{0}(f)$ is essentially surjective. One has to prove that the $f$ is locally a quasi-isomorphism: $f(a,b): \X(a,b) \rightarrow \Y(a,b), \forall a,b \neq 1_{\K}$. Denote $h(0_{A}) = u$. \\
	By Theorem \ref{firstbigthm}, we know that $\K$ is a resolution of the $\Bbbk-$linear envelope of the ordinary category with two objects $0,1$ and with only one morphism between any pair of  objects. In particular, $\K(0,0)$ is quasi-isomorphic to $\Bbbk[0]$.
	Therefore
	\begin{equation}
\bar{\K}:=\K(0,0)/\Bbbk[0]
	\end{equation}
	 is a complex acyclic in all degrees. \\
	By Proposition \ref{propappc}, we have: 
	\begin{equation}
	\begin{aligned}
	\Y(a,b) := \;\X(a,b) &\bigoplus\mathcal{O}(3) \otimes \X(u,b)\otimes \bar{\K} \otimes \X(a,u) \\
	&\bigoplus \mathcal{O}(5) \otimes \X(u,b)\otimes \bar{\K} \otimes \X(u,u)\otimes \bar{\K} \otimes \X(a,u) \bigoplus\dots
	\end{aligned}
	\end{equation}
It is a direct sum of complexes, among which all but the first one are acyclic, due to the acyclicity of $\bar{\K}$. It completes the proof that $f$ is a quasi-equivalence. 
\end{proof}
Theorem \ref{bigtheorem} is proven.
\end{proof}
\comment
\begin{theorem}[Model structure on $\Cat^0_{\dgwu}(\Bbbk)$]\label{cor.mod}
The category $\Cat^0_{dgwu}(\Bbbk)$ admits a cofibrantly generated closed model structure whose weak equivalences and fibrations are defined above, and whose sets of generating cofibrations and generating acyclic cofibrations are those $I$ and $J$, where we consider refined weakly unital dg categories and functors, instead of weakly unital.
\end{theorem}
\begin{proof}
The same arguments work.
\end{proof}
\endcomment

\section{\sc A Quillen equivalence between $\Cat_{\dgwu}(\Bbbk)$ and $\Cat_{\dg}(\Bbbk)$}\label{s.quieq}
Let $\C$ and $\D$ be two model categories. Recall that a \emph{Quillen pair} of functors $L: \C \leftrightarrows \D: R$ is an adjoint pair of functors such that $L$ preserves cofibrations and acyclic cofibrations, or equivalently, $R$ preserves fibrations and acyclic fibrations, [Ho, 1.3], [Hi, 8.5]. These two conditions are sufficient to show that the Quillen pair of functors descends to a pair of adjoint functors 
\begin{equation}\label{adjho}
L: Ho(\C) \leftrightarrows Ho(\D): R
\end{equation} 
between the homotopy categories. \\
When $\C$ is cofibrantly generated, there is a manageable criterium for an adjoint pair of functors to be a Quillen pair [\cite{Ho}, Lemma 2.1.20]:
\begin{prop}\label{propHovey}
Let $\C, \D$ be closed model categories, with $\C$ cofibrantly generated with generating cofibrations $I$ and generating acyclic cofibrations $J$. Let $L: \C \leftrightarrows \D: R$ be an adjoint pair of functors. Assume that $L(f)$ is a cofibration for all $f \in I$ and $L(f)$ is a trivial cofibration for all $f \in J$. Then $(L,R)$ is a Quillen pair.
\end{prop}

Let $C \in \Cat_{\dgwu}(\Bbbk)$. Define 
\[ L_1(C) := C/I, \]
where $I$ is the dg-category ideal generated by all $p_n(..), n \geq 2.$ Clearly $L_1(\C) \in \Cat_{\dg}(\Bbbk)$. This assignment $C \mapsto L_1(\C)$ gives rise to a functor $L_1: \Cat_{\dgwu}(\Bbbk) \rightarrow \Cat_{\dg}(\Bbbk)$. \\
Let $\D \in \Cat_{\dg}(\Bbbk)$. Define  $R_1: \Cat_{\dg}(\Bbbk) \rightarrow \Cat_{\dgwu}(\Bbbk)$
as the fully-faithful embedding from Example \ref{exwusu}. 

\begin{prop}\label{propqp}
The following statements are true:
\begin{itemize}
\item[(1)] there is an adjunction $$ \Hom_{\Cat_{\dg}(\Bbbk)}(L_1(\C), \D) \cong \Hom_{\Cat_{\dgwu}(\Bbbk)}(\C, R_1(\D)) $$
\item[(2)] the functors $$L_1: \Cat_{\dgwu}(\Bbbk) \leftrightarrows \Cat_{\dg}(\Bbbk): R_1 $$
form a Quillen pair of functors.
\end{itemize}
\end{prop}
\begin{proof}
$(1):$ any morphism $F: C \rightarrow R_1(D)$ sends $p_n(..), n \geq 2$ to $0$, since $D \in \Cat_{\dg}(\Bbbk)$, and therefore this morphism is the same as a morphism $L_1(C) \mapsto D$. \\
$(2):$ the dg categories $\{ L_1(\beta(n)), L_1(Q)\}$ form exactly the set $I$ of generating cofibrations for the Tabuada closed model structure on $\Cat_{\dg}(\Bbbk)$, and the dg categories $\{ L_1(\alpha(n)), L_1(\kappa)\}$ form the set of generating acyclic cofibrations for this model strcuture. The statement follows then by Proposition \ref{propHovey}.
\end{proof}

Recall that a Quillen pair $L: \C \leftrightarrows \D: R$ is called a \emph{Quillen equivalence} if it satisfies the following condition: \\
For a cofibrant $X \in \C$ and a fibrant $Y \in \D$, a morphism $f: LX \rightarrow Y$ is a weak equivalence in $\D$ if and only if the adjoint mo, [Ho, 1.3.3], [Hi, 8.5]. This condition implies the corresponding adjoint pair of functors between the homotopy categories \ref{adjho} is an adjoint \emph{equivalence} of categories.

\begin{theorem}\label{quieq}
The Quillen pair of functors
\[ L_1: \Cat_{\dgwu}(\Bbbk) \leftrightarrows \Cat_{\dg}(\Bbbk): R_1 \]
is a Quillen equivalence.
\end{theorem}
\begin{proof}\label{proof.quieq}
Let $X \in \Cat_{\dgwu}(\Bbbk)$ be cofibrant and $Y \in \Cat_{\dg}(\Bbbk)$ fibrant. We have to prove that $f: L_1X \rightarrow Y$ is a weak equivalence in $\Cat_{\dg}(\Bbbk)$ if and only if the adjoint morphism $f^{*}: X \rightarrow R_1Y$ is a weak equivalence in $\Cat_{\dgwu}(\Bbbk)$, \\
It is enough to prove the statement for the case when $X$ is an $I$-cell. Indeed, by the small object argument, for any $X$ there exist an $I$-cell $X^\prime$ such that $p\colon X^\prime\to X$ is an acyclic fibration. The Quillen left adjoint $L$ maps the weak equivalences between cofibrant object to weak equivalences, by [Hi, Prop. 8.5.7]. Therefore, $L(p)\colon L(X^\prime)\to L(X)$ is a weak equivalence. There is a map $i\colon X\to X^\prime$ such that $p\circ i=\id$, given by the RLP. By 2-of-3 axiom, $i$ is a weak equivalence, and $L(i)$ also is. \\
We assume that $X$ is an $I$-cell. 
Denote by $V$ the graded graph of generators of $X$. We need to prove that for any objects $x,x' \in X$, the cone $M := Cone(f: L_1X(x,x') \rightarrow Y(fx,fx'))$ is acyclic if and only if the cone $N := Cone(f^{*}: X(x,x') \rightarrow R_1Y(f^{*}x,f^*x'))$. Denote by $\tilde{\mathcal{O}} := Ker(P: \mathcal{O}' \rightarrow \mathcal{A}ssoc_+)$, where $P$ is the dg operad map sending all $p_n(...), n\geq 2$ to $0$. \\
There is a canonical map $\omega: N \rightarrow M$, and $Cone(\omega)$ is quasi-isomorphic to $F_{\tilde{\mathcal{O}}}(V)(x,x')$, where $F_{\tilde{\mathcal{O}}}(V)$ is the free algebra over $\tilde{\mathcal{O}}$ generated by $V$, with an extra differential coming from the differential in the $I-$cell $X$. Since the dg operad $\mathcal{O}'$ is quasi isomorphic to $\mathcal{A}ssoc_{+}$, $\tilde{\mathcal{O}}$ is acyclic. Therefore $F_{\tilde{\mathcal{O}}}(V)$ is acyclic by the K\"{u}nneth formula, and so $M$ is quasi isomorphic to $N$, by the acyclicity of $Cone(\omega)$. We conclude that $M$ is acyclic if and only if $N$ is.
\end{proof}

\comment
In our previous work [\cite{PS}, Prop 3.2, Thm 3.3] we proved that 
\begin{equation}\label{old.adj}
L: \mathcal{C}at^{'}_{dgwu}(\Bbbk) \leftrightarrows \mathcal{C}at_{dg}(\Bbbk): R 
\end{equation}
is a Quillen pair of functors, and moreover it is a Quillen equivalence.
Therefore we have:
\begin{theorem}
There is an equivalence of categories 
\begin{equation}\label{adj}
Ho(\mathcal{C}at_{dgwu}(\Bbbk)) \cong Ho(\mathcal{C}at^{'}_{dgwu}(\Bbbk))
\end{equation}
\end{theorem}
\begin{proof}
By Theorem \ref{quieq}, we know that $Ho(\mathcal{C}at_{dgwu}(\Bbbk))$ is equivalent to $Ho(\mathcal{C}at_{dg}(\Bbbk))$, and since the adjunction \ref{old.adj} is also a Quillen equivalence, it follows that $Ho(\mathcal{C}at_{dgwu}(\Bbbk)) \cong Ho(\mathcal{C}at^{'}_{dgwu}(\Bbbk)).$ \\
An alternative proof follows by considering the model category structure given in \ref{cor.mod}, and similarly define a Quillen adjunction $L': \mathcal{C}at_{dgwu}(\Bbbk) \leftrightarrows \mathcal{C}at^{'}_{dgwu}(\Bbbk): R'$, which can be proven mutatis mutandis as in proof of \ref{quieq} to be a Quillen equivalence.
\end{proof}
\endcomment

\comment

\section{\sc An application to computation of $\Hot(C,D)$ and $\Hot^{\Delta}(C,D)$}

\subsection{\sc }
Let $\mathcal{M}$ be a closed model category, denote by $\Ho(\mathcal{M})$ its homotopy category. In order to compute the Hom-sets $\Ho(\mathcal{M})(X,Y)$, one has to replace $X$ by any cofibrant $r(X)$ such that there is a trivial fibration $r(X)\to X$, to replace $Y$ by any fibrant $i(Y)$ such that there is a trivial cofibration $Y\to i(Y)$, and to compute the set 
$$
\Ho(\mathcal{M})(X,Y):=\mathcal{M}(r(X),i(Y))/\sim
$$
where $\sim$ is an equivalence relation, derived from either a path object of $i(Y)$ or a cylinder object of $r(Y)$, see e.g. ???.

Trying to apply this recipy to the category $\Cat_\dg(\k)$ of {\it strictly} unital dg categories over $\k$, where $\k$ is a field, we experience the following difficulty. 

At first, there is no problems with $i(Y)$, any object in $\Cat_\dg(\k)$ is fibrant, so we can simply take $i(Y)=Y$.
On the other hand, the most natural explicit choice for $r(X)$, equal to $\Cobar(\Bar(X))$, is not a strictly unital dg category. To explain it, consider the simplest case of a dg category with a single object, that is, of a dg algebra over $\k$.

Let $A$ be a dg algebra, {\it not unital}. Recall that $\Bar(A)$ is the coaugmented {\it counital} dg coalgebra defined as the sum total complex of the bicomplex
$$
\dots\to \underset{\deg=-3}{A^{\otimes 3}}\to \underset{\deg=-2}{A^{\otimes 2}}\to \underset{\deg=-1}{A}\xrightarrow{0} \underset{\deg=0}{\k}\to 0
$$
where the differential is the bar-differential. One has:
\begin{equation}\label{dbar}
\begin{aligned}
\ &d_\Bar(a_1\otimes a_2)=a_1\cdot a_2\\
&d_\Bar(a_1\otimes a_2\otimes a_2)=(a_1\cdot a_2)\otimes a_3-a_1\otimes (a_2\otimes a_3)\\
&\dots
\end{aligned}
\end{equation}

Let $C$ be a coaugmented counital dg coalgebra over $\k$. Its cobar-complex is the {\it non-unital} dg algebra equal to the total sum complex of the bicomplex
$$
0\to \underset{\deg=1}{\overline{C}}\to \underset{\deg=2}{\overline{C}^{\boxtimes 2}}\to\underset{\deg=3}{\overline{C}^{\boxtimes 3}}\to \dots
$$
where $\boxtimes$ is $\otimes_\k$, as well as $\otimes$, but we use this notation to distinguish the tensor products in the two complexes, $\overline{C}=C/\k$, and the differential is the cobar-differential:
\begin{equation}\label{dcobar}
\begin{aligned}
\ &d_\Cobar(c)=\Delta(C)\\
&d_\Cobar(c_1\boxtimes c_2)=\Delta(c_1)\boxtimes c_2-c_1\boxtimes \Delta(c_2)\\
&\dots
\end{aligned}
\end{equation}
where $\Delta\colon C\to C\boxtimes C$ is the coproduct. 

The functors $\Bar$ and $\Cobar$ are adjoint on suitable categories (the category of non-unital dg algebras and the category of ind-conilpotent coaugmented dg coalgebras), with $\Cobar$ the left adjoint. 
It is well-known that the natural map of dg algebras (given by the adjunction unit)
$$
p\colon \Cobar(\Bar(A))\to A
$$
is a quasi-equivalence of {\it non-unital} dg algebras. 

Assume now that the dg algebra $A$ is unital, one may ask oneself whether it is true that $r(A)=\Cobar(\Bar(A))$ is also unital. 

There is a natural ``unit element'' $1=1_A\in (A^{\boxtimes 1})^{\otimes 1}\subset \Cobar(\Bar(A))$. However, it is only a {\it weak unit}.

For example, take any $a\in (A^{\boxtimes 1})^{\otimes 1}$ with $d_A(a)=0$, one has:
$$
1\boxtimes a-a=d_\tot(1\otimes a)
$$
where $d_\tot=d_\Cobar+d_\Bar$, and $1\otimes a\in (A\otimes A)^{\boxtimes 1}$. Similarly,
$$
a\boxtimes 1-1=d_\tot(a\otimes 1)
$$
In fact, one can endow $\Cobar(\Bar(A))$ with Kontsevich-Soibelman wu structure, see [PS, ???].
{\it However, this wu dg coalgebra is not an $I\text{-}cell$, and is not cofibrant}.
(It is cofibrant as a non-unital dg algebra, of course). 

We conclude that the cobar-bar resolution does not belong to the category of strictly unital dg algebras, and, as an object of the category of weakly unital dg algebras, it is not cofibrant.

\subsection{\sc }
Denote by $\Hot_\dg(\k)$, or shortly $\Hot_\dg$, the homotopy category of the closed model category $\Cat_\dg(\k)$, and denote by $\Hot_\dgwu(\k)$, or shortly $\Hot_\dgwu$, the homotopy category of the closed model category $\Cat_\dgwu(\k)$.

We want to compute $\Hot_\dg(C,D)$.\footnote{To the best of our knowledge, the existing proof is based on the description of $\Hot$ via quasi-functors, given in [To]. The approach we suggest here is more elementary and direct.} To this end, we find an explicit cofibrant replacement $r(C)$ of $C$, which is a {\it weakly unital dg category} (and is cofibrant in $\Cat_\dgwu(\k)$).\footnote{A cofibrant replacement of $C$ in $\Cat_\dg(\k)$ exists as well, but any construction uses the Quillen small object argument, and thus is not canonical and less manageable than our $r(C)$ in $\Cat_\dgwu(\k)$.}

We employ Theorem \ref{quieq} by which
\begin{equation}
\Hot_\dg(L_1(C), D)\simeq \Hot_\dgwu(C,R_1(D))
\end{equation}
(for $C$ cofibrant in $\Cat_\dgwu(\k)$ and $D$ fibrant in $\Cat_\dg(\k)$).

Any $D\in \Cat_\dg(\k)$ is fibrant, and thus $R_1(D)$ also is. 
For a dg category $C$, we construct its cofibrant replacement $r(C)$ in $\Cat_\dgwu(\k)$, such that $L_1(r(C))=C$.
Then we have, for {\it any} $C,D\in \Cat_\dg(\k)$:
$$
\Hot_\dg(C,D)\simeq \Hot_\dgwu(r(C), R_1(D))
$$
Our resolution $r(C)$ is a fattened version of $\Cobar(\Bar(C))$, considered above.
It is referred to as {\it the fattened cobar-bar resolution}. 

\subsection{\sc The fattened cobar-bar resolution}

\endcomment

\comment
Let $A$ be a coaugmented counital dg cocategory. 
Denote by $\overline{A}$ the dg graph having the same objects as $A$ and whose morphisms are defined as
$$
\overline{A}(x,y)=\begin{cases}
A(x,y)&x\ne y\\
A(x,x)/\k\id_x&x=y
\end{cases}
$$
Consider the free wu dg category generated the dg graph $\overline{A}[-1]$, denote it $\Free_{\mathcal{O}^\prime}(\overline{A}[-1])$.

It is known that the free {\it unital} dg category generated by $\overline{A}[-1]$ admits the {\it cobar differential}, see \eqref{dcobar}. Our goal is to prove that a similar differential is defined on $\Free_{\mathcal{O}^\prime}(\overline{A}[-1])$.

\begin{lemma}\label{lemmafcb1}
There is a differential on $\Free_{\mathcal{O}^\prime}(\overline{A}[-1])$, called the cobar-differential $D_\Cobar$, which makes it a cofibrant wu dg category. 
\end{lemma}
\begin{proof}
We define $D_\Cobar$ as the sum 
\begin{equation}\label{dcobarD}
D_\Cobar=d_\Cobar+d_\wu
\end{equation}
where $d_\Cobar$ is defined by \eqref{dcobar} and is extended by the Leibniz rule with respect to all $p_n(\dots)$, and $d_\wu$ is the sum of the underlying differential on $A$ and the differential $dp_n(\dots)$ (likewise \eqref{eqp2}, \eqref{eqp3}). It is clear that $d_\Cobar^2=0$ and $d_\wu^2=0$, one has to show that 
\begin{equation}\label{dcobarD2}
d_\Cobar d_\wu+d_\wu d_\Cobar=0
\end{equation}
One sees \eqref{dcobarD2} directly, which gives $D_\Cobar^2=0$. 
\end{proof}

Let $C$ be a strictly unital dg category, then $A:=\Bar(C)$ is a counital coaugmented dg coalgebra. 

\begin{lemma}\label{lemmafcb2}
There is a weak equivalence 
$$
p\colon \Free_{\mathcal{O}^\prime}(\overline{\Bar(C)}[-1])\to R_1(C)
$$
One has
$$
L_1(\Free_{\mathcal{O}^\prime}(\overline{\Bar(C)}[-1]), D_\Cobar)=\Cobar(\Bar(C))
$$
\end{lemma}
\begin{proof}
We have a quasi-isomorphism of dg operads $P\colon \mathcal{O}^\prime\to \Assoc_+$ sending all $p_n(\dots), n\ge 2$ to 0. Thus we have an induced map 
$$
P_*\colon (\Free_{\mathcal{O}^\prime}(\overline{\Bar(C)}[-1]), D_\Cobar)\to
(\Free_{\Assoc_+}(\overline{\Bar(C)}[-1]), d_\Cobar+d_\Bar+d_C)
$$
which agrees with the differentials and is a quasi-isomorphism. 

On the other hand, $\Free_{\Assoc_+}(\overline{\Bar(C)}[-1])$ is isomorphic to $\Cobar(\Bar(C))$. The latter dg category is quasi-isomorphic to $C$ by the classical bar-cobar duality. 

The last statement follows directly from the definition of the functor $L_1$.
\end{proof}
\endcomment

\comment
Here we construct a ``fattened'' version of the cobar-resolution of $A$, where $A$ is a {\it unital} dg category over $\k$. To simplify our notations, below we assume that $A$ is a unital dg algebra.
Denote by $1_A$ the unit of $A$. 

We start with a ``fattened'' version of the bar-complex.

Introduce symbols $\xi^k_{x_1,\dots,x_n}$, $n\ge k$, $k\ge 0$, in which $n-k$ of $x_i$s are elements of $A$, and the remaining $k$ among them are {\it formal units},denoted by $1$. Here formal 1s are just formal symbols, {\it not} equal to $1_A$. We assume that
\begin{itemize}
	\item[(a)] the symbols $\xi^k_{x_1,\dots,x_n}$ are $\k$-polylinear in the $n-k$ elements of $A$,
	\item[(b)] $\xi^0_{x_1,\dots,x_n}=x_1\otimes\dots\otimes x_n\in A[1]^{\otimes n}$,
	\item[(c)] $\deg \xi^k_{x_1,\dots,x_n}=\sum\deg x_i-k-n$, where $\deg (1)=0$,
	\item[(d)] $\xi^k_{x_1,\dots,x_n}=0$ if $\sharp\{i|\ x_i=1\}\ne k$.
\end{itemize}
Thus, we can identify the $\k$-vector space of all $\xi^k_{x_1,\dots,x_n}$, for a given $k\ge 0$, with 
$$
\bigoplus_{j\colon \{1,2,\dots,k\}\to \{1,2,\dots,n\}}A[1]^{\otimes(n-k)}[k]=\bigoplus_{j\colon \{1,2,\dots,k\}\to \{1,2,\dots,n\}}A^{\otimes (n-k)}[n]
$$
where the sum is taken over all {\it imbeddings} $j$. (The image of $j$ consists of those indices $i$ for which $x_i=1$).
We denote this vector space by $\widehat{\Bar}_k(A)$.

Denote 
$$
\widehat{\Bar}(A)=\bigoplus_{k\ge 0}\widehat{\Bar}_k(A)
$$
Note that $\widehat{\Bar}_0(A)$ coincides with the underlying graded vector space of the bar-complex \eqref{dbar}.

We endow $\widehat{\Bar}(A)$ with a dg coalgebra structure such that the underlying complex of $\widehat{\Bar}(A)$ is the direct sum 
$$
\widehat{\Bar}(A)=\bigoplus_{k\ge 0}\widehat{\Bar}_k(A)\text{\it \ \ as a complex}
$$
and such that $\widehat{\Bar}_0(A)$ is a dg sub-coalgebra, which coincides with the classical bar-dg-coalgebra. 

The complex $\widehat{\Bar}_k(A)$ is defined by
\begin{equation}\label{dbark}
\begin{aligned}
d_\Bar(\xi^k_{x_1,\dots,x_n})=&d_A+\sum_{i=1}^{n-1}\pm\xi^k_{x_1,\dots,x_{i-1},x_ix_{i+1},\dots,x_n}+\sum_{i|x_i=1}\pm \xi^{k-1}_{x_1,\dots,x_{i-1}, 1_A, x_{i+1},\dots,x_n}
\end{aligned}
\end{equation}
where $d_A$ is the component obtained by extension of the differential in $A$ by the Leibniz rule.
For the second summand, only the summands for which both  $x_i$ or $x_{i+1}$ are not equal to 1 do contribute (formally we allow all summands, adapting the convention $x_i 1=1 x_i=1$, and use that $\xi^k_{x_1,\dots,x_\ell}=0$ if $\sharp\{i| x_i\ne 1\}\ne k$).  

One computes that $d_\Bar^2=0$. 

Now we endow $\widehat{\Bar}(A)$ with a dg coalgebra structure by
\begin{equation}\label{deltabarhat}
\Delta(\xi^k_{x_1,\dots,x_n})=\oplus_{a+b=k}\oplus_{n_1+n_2=n}
\xi^a_{x_1,\dots,x_{n_1}}\otimes \xi^b_{x_{n_1+1},\dots,x_n}
\end{equation}
where only the summands for which 
$\sharp\{1\le i\le n_1|\ x_1=1\}=a$ (and thus $\sharp\{n_1+1\le i\le n|\ x_i=1\}=b$) do contribute to the r.h.s.

\begin{lemma}\label{lemmabarhat}
Formulas \eqref{dbark} and \eqref{deltabarhat} define a dg coalgebra structure on $\widehat{\Bar}_k(A)$. 	
\end{lemma} 
\qed

Consider
\begin{equation}\label{defcobarplus}
\Cobar_+(\hat{\Bar}(A)):=\bigoplus_{\ell\ge 1}\widehat{\Bar}(A)[-1]^{\otimes \ell}\oplus \k=\Free_{\Assoc_+}(\widehat{\Bar}(A)[-1])
\end{equation}
We denote by $\otimes$ the tensor product involved in the definition of $\widehat{\Bar}(A)$, and by $\boxtimes$ the tensor product in \eqref{defcobarplus}. 

Denote  $1_\k\in\k\subset \Cobar_+(\hat{\Bar}(A))$. 
Endow $\Cobar_+(\hat{\Bar}(A))$ with the differential defined on generators as
\begin{equation}\label{dtotbarcobar}
d_\tot=d_\Bar+d_\Cobar
\end{equation}
where 
\begin{equation}
\begin{aligned}
d_\Cobar(\xi^k_{x_1,\dots,x_n})=&\pm \Delta(\xi^k_{x_1,\dots,x_n}),\ \ n\ge 2\\
d_\Cobar(\xi^1_1)=&1_\k
\end{aligned}
\end{equation}
and extended by the Leibniz rule. 

Here are some examples:

\begin{equation}
d\xi^1_1=1-1_A
\end{equation}
\begin{equation}
d\xi^1_{1,a}=a\otimes 1_A-\xi_1^1\boxtimes a
\end{equation}
\begin{equation}
d\xi^1_{a,1}=a\otimes 1_A-a\boxtimes \xi^1_1
\end{equation}
\begin{equation}
d\xi^1_{a,b,1}=a\otimes b\otimes 1_A\pm \xi^1_{ab,1}\pm a\boxtimes \xi^1_{b,1}\pm (a\otimes b)\boxtimes \xi^1_1
\end{equation}
\begin{equation}
d\xi^2_{1,a,1}=\xi^1_{1_A,a,1}\pm \xi^1_{1,a,1_A}\pm \xi^1_1\boxtimes \xi^1_{a,1}\pm \xi^1_{1,a}\boxtimes \xi^1_1
\end{equation}

\begin{lemma}
One has $d_\Bar^2=0$, $d_\Cobar^2=0$, $d_\Bar d_\Cobar+d_\Cobar d_\Bar=0$, which makes $\Cobar_+(\widehat{\Bar}(A))$ a bicomplex.
\end{lemma}
\qed

\begin{prop}\label{propbarhat}
Endowed with the differential \eqref{dtotbarcobar}, $\Cobar_+(\hat{\Bar}(A))$ is quasi-isomorphic to $A$, as a {\bf unital} dg algebra. 
\end{prop}
\begin{proof}
	
\end{proof}

\begin{remark}{\rm
Note a difference between the statement of Proposition 	\ref{propbarhat} and the classical statement $\Cobar(\Bar(A))\overset{quis}{\sim}A$. In the latter statement, $\Cobar(\Bar(A))$ is a {\it non-unital} dg algebra, whence our $\Cobar_+(\widehat{\Bar}(A))$ is a {\it unital} dg algebra. Thus, Proposition \ref{propbarhat} provides a cofibtant resolution of $A$ as a {\it unital} dg algebra, for a unital dg algebra $A$. 
}
\end{remark}

\subsection{\sc Computation for $\Hot(C,D)$}

\subsection{\sc Computation for $\Hot^\Delta(C,D)$}

\endcomment

\section{\sc A proof of Proposition \ref{propph}}
Here we prove Proposition \ref{propph}.\\
Recall the maps $P_n$, $n\ge 1$ defined in \eqref{bigP}.\\
Recall that  by Lemma \ref{lemmaph} one has $$deg(P_n(\phi_1, .., \phi_n))_{ij} = |\phi_1| + .. + |\phi_n| - n + 1 = deg (p_l(\kappa_1, .., \kappa_l)), \forall \kappa \in \pa_{ij}$$

Proposition \ref{propph} reads:
\begin{prop}
	The maps $\{P_n\},\ n\ge 1$ are Taylor components of an $A_\infty$ functor  $$P\colon {C}^{\pretr} \oplus \Bbbk_{{C}^{\pretr}}\to {C}^{\pretr}$$ 
\end{prop}

\begin{proof}
Proving the statement amounts to proving the following identies, for any $n\ge 1$ (see \eqref{ainftysigns} for the sign convention):
	\begin{equation}\label{eq.fund}
	\begin{aligned}
&\ 	d P_n(\phi_n,\dots,\phi_1) +
\sum_{a+b = n} (-1)^{(a-1)(\sum_{i=1}^b|\phi_i|)+b-1} P_a(\phi_n,\dots,\phi_{b+1})P_{b}(\phi_b,\dots,\phi_1)=\\
&\sum_{k=1}^{n-1} (-1)^{k-1} P_{n-1}(\phi_n,\dots, \phi_{k+1}\circ \phi_{k},\dots,\phi_1) + \sum_{k=1}^{n} (-1)^{n-1+\sum_{i=1}^{k-1}|\phi_i|} P_n(\phi_n,\dots, d\phi_k, \dots,\phi_1),
	\end{aligned}
	\end{equation}
	where 
	\begin{equation}
	dP_n(\phi_n,\dots,\phi_1)= d_{naive}P_n(\phi_n,\dots,\phi_1) + q'\circ P_n(\phi_n,\dots,\phi_1) - (-1)^{n-1} P_n(\phi_n,\dots,\phi_1)\circ q
	\end{equation}

	Writing down explicitly the first line of equation \ref{eq.fund} (and dropping the signs to $\pm$ for simplicity), we get 
	\begin{equation}\label{eq*}
	\begin{aligned}
	\ &d P_n(\phi_n,\dots,\phi_1)=d_{naive}P_n(..) + q'\circ P_n(..) +(-1)^{n-1} P_n(..)\circ q =\\ 
	&d_{naive}\Bigl(\sum_{\kappa \in \pa} \pm p_m(\kappa_m, .., \kappa_1)\Bigr) + q'\circ \sum_{\kappa \in \pa}\pm p_{m'}(\kappa_{m'}, .., \kappa_{1}) +(-1)^{n-1} \sum_{\kappa \in \pa}p_{m''}(\kappa_{m''}, .., \kappa_{1})\circ q = \\
	&\sum_{\kappa \in \pa} \Bigl(\sum_{a+b = m} \pm p_a(..)\circ p_{b}(..) + \sum_{i=1}^{n-1} \pm p_{m-1}(.., m(\kappa_{i+1},\kappa_{i}), ..) + \sum_{i=1}^{m} \pm p_n(.., d\kappa_i, ..)\Bigr)+ \\
	&q'\circ \sum_{\kappa \in \pa}\pm p_{m'}(\kappa_{m'}, .., \kappa_{1}) +(-1)^{n-1} \sum_{\kappa \in \pa}\pm p_{m''}(\kappa_{m''}, .., \kappa_{1})\circ q
	\end{aligned}
	\end{equation}
	We stress that inside the terms $\sum_{\kappa \in \pa} \left(\sum_{i=1}^{n-1} (-1)^{s} p_{m-1}(.., m(\kappa_{i+1},\kappa_{i}), ..)\right)$ the composition term might be of following three types:
	\begin{itemize}
		\item $q_{jl}\circ q_{kj}$ 
		\item $q\circ\phi$ or $\phi\circ q$
		\item $\phi_{i}\circ\phi_{i+1}$.
	\end{itemize}
	Similarly inside the terms $\sum_{\kappa \in \pa}\left( \sum_{i=1}^{m} (-1)^{t} p_n(.., d\kappa_i, ..) \right)$ the term which is differentiated might be of the following two types:
	\begin{itemize}
		\item $d\phi$
		\item $dq_{kl}$.
	\end{itemize}
	
We also write down the possible terms of $\sum_{\kappa \in \pa}\left( \sum_{a+b = m} \pm  p_a(..)p_{b}(..) \right)$:
	\begin{itemize}
		\item $p_a(..)p_b(..)$
		\item $q\circ p_{m}$ or $p_l(..) \circ q$.
	\end{itemize}
	
	Thence we can write for the r.h.s. of \eqref{eq*}
	\begin{equation}
	\begin{aligned}
\ &	(\text{r.h.s. of }\eqref{eq*}) = \\
&\sum_{\kappa \in \pa} \pm \Bigl( \uline{\sum_{i} \pm p_n(.., d\phi, ..)} \pm \dashuline{\sum_{i}\pm p_n(.., dq_{kl}, ..)} \pm \dashuline{\sum_{i} \pm p_n(.., q_{jl}\circ q_{kj}, ..)}  \Bigr)+ \\
	& \sum_{\kappa \in \pa}\pm  \Bigl( \pm \uline{\sum_{i} \pm  p_n(.., q\circ\phi, ..)} \pm \uline{\sum_{i} \pm p_n(.., \phi\circ q, ..)} 
	\pm \dotuline{\sum_{i} \pm p_n(.., \phi_{i+1}\circ\phi_{i}, ..)} \Bigr)+ \\ 
	&\sum_{\kappa \in \pa}\pm \Bigl( \uwave{\pm q' \circ p_m(..)} \pm \uuline{\pm  p_l(..)\circ q} \pm \sum_{a} \pm  p_a(..)\circ p_b(..) \Bigr)+ \\ 
	&\uwave{q'\circ \sum_{\kappa \in \pa} \pm p_{m'}(\kappa_1, .., \kappa_{m'})}+ \uuline{\sum_{\kappa \in \pa} \pm p_{m''}(\kappa_1, .., \kappa_{m''})\circ q} \overset{?}{=}\\
	&\sum_{a+b=n}\pm P_a(..)\circ P_{b}(..) + \dotuline{\sum_{i}\pm P_{n-1}(.., m(\phi_{i+1},\phi_{i}), ..)} + \uline{\sum_{i}\pm P_n(.., d\phi, ..)}
	\end{aligned}
	\end{equation}
	We have to prove the equation marked by question sign. 
	The \dashuline{dashed underlined} terms gets cancelled by the Maurer-Cartan condition on the $q_{ij}$: indeed if there exists $\kappa \in \pa_{ij}$ which contains $q_{kl}$, there will also exist a path $\kappa' \in \pa_{ij}$ containing all the terms $q_{jl}\circ q_{kj}$, since they have got same domain and codomain as $q_{kl}$ and in $\pa_{ij}$ we were considering all the possible paths. \\
	Therefore we are left with the following expressions:
	\begin{equation*}
	\begin{aligned}
	&\ \pm \sum_i  \sum_{\kappa \in \pa} p_n(.., m(\phi_{i+1},\phi_{i}), ..)  \pm \sum_{a+b=n}  \sum_{\kappa \in \pa}\pm  p_a(...)\circ p_b(...) + \\
	&\sum_{i} \sum_{\kappa \in \pa}\pm \bigl(p_n(.., d\phi_i,..) + p_n(.., q'\circ\phi_i,..) + p_n(.., \phi_i\circ q,..))  \bigr) = \\
	&\sum_{a+b=n}\pm P_a(..)\circ P_{b}(..) + \sum_{i}\pm  P_{n-1}(.., m(\phi_{i+1},\phi_{i}), ..) + \sum_{i}\pm P_n(.., d\phi, ..),
	\end{aligned}
	\end{equation*}
	which shows the desired equation, up to signs. The correctness of signs (which were explicitely defined in \eqref{bigP} and \eqref{eq.fund}) is checked by a long but routinous computation.
\end{proof}

\comment
Therefore, we defined an $A_\infty$-map $P$ from the strict unital dg-category $\mathcal{C}^{pretr} \oplus \Bbbk_{\mathcal{C}^{pretr}}$ to the non unital dg category $\mathcal{C}^{pretr}.$
In particular, it simply follows that: if $\C \in \mathcal{C}at_{dgwu}(\Bbbk)$, then $\C^{pretr} \in \mathcal{C}at_{dgwu}(\Bbbk)$, whereas if $\C \in \mathcal{C}at^{'}_{dgwu}(\Bbbk)$, then $\C^{pretr} \in \mathcal{C}at^{'}_{dgwu}(\Bbbk).$
\endcomment

\appendix
\section{}

\comment
In this appendix we will remind the reader about the notions of $A_{\infty}$-category and of $A_{\infty}$-functor, in order to have a convenient reference for signs.

\begin{defn}[$A_{\infty}$-category]
An $A_{\infty}$-category $A$ is the data of: 
\begin{itemize}
\item A set of objects $Ob(A)$;
\item for any pair of objects $x,y \in Ob(A)$, a $\mathbb{Z}-$graded vector space over $\Bbbk$, denoted by $Hom_{A}(x,y)$;
\item for $n\geq1$ and each sequence of objects $x_0, .., x_n$, a degree $2-n$ map 
\begin{equation*}
m_n: Hom_{A}(x_{n-1},x_n) \otimes \cdot \cdot \otimes Hom_{A}(x_{0},x_1) \rightarrow Hom_{A}(x_{0},x_n)
\end{equation*}
such that, for every $n\geq1$
\begin{equation}
\sum_{n=r+t+s} (-1)^{rs+t}m_{r+t+1}(id^{\otimes^{r}}\otimes m_s \otimes id^{\otimes^{t}}) = 0.
\end{equation}
\end{itemize}
\end{defn}

\begin{defn}[$A_{\infty}$-functor]
Let $A, \mathcal{B}$ be two $A_{\infty}$-categories. An $A_{\infty}$-functor $f: A \rightarrow \mathcal{B}$ is the data of:
\begin{itemize}
\item A map of sets $f_0: Ob(A) \rightarrow Ob(\mathcal{B})$
\item for $n\geq1$ and each sequence of objects $x_0, .., x_n$, a degree $1-n$ map
\begin{equation*}
f_n: Hom_{A}(x_{n-1},x_n) \otimes \cdot \cdot \otimes Hom_{A}(x_{0},x_1) \rightarrow Hom_{\mathcal{B}}(f_0(x_{0}),f_0(x_n))
\end{equation*}
such that, for every $n\geq1$
\begin{equation*}
\sum_{n=r+t+s}(-1)^{rs+t}f_{r+1+t}(id^{\otimes^{r}}\otimes m_s \otimes id^{\otimes^{t}}) = \sum_{\substack{1\leq r\leq n \\ i_1 + \cdot\cdot + i_r = n}}(-1)^{\epsilon_r}m'_r(f_{i_1} \otimes \cdot \cdot \otimes f_{i_r})
\end{equation*}
where 
\begin{equation*}
\epsilon_r = \epsilon_r(i_1, .., i_r) = \sum_{2\leq k \leq r} \left( (1-i_k) \sum_{1\leq l \leq k-1}i_l \right)
\end{equation*}
\end{itemize}
\end{defn}

\endcomment

In the paper, we use only $A_\infty$ functors between {\it dg categories}.

We need to specify the signs in the $A_\infty$ identity. The common sign conventions can be found in [Kel, 3.4] for the {\it left to right} order formulas, and in [Ly1, 2.4] for the {\it right to left} order formulas. As we adapt here the right to left formalism, our signs agree with the ones in [Ly1].

\begin{defn}{\rm
Let $C,D$ be (non-unital) dg categories over $\k$. An $A_\infty$ functor $F\colon C\to D$ is given by: 
\begin{itemize}
	\item a map of sets $F_0\colon \Ob(C)\to\Ob(D)$,
	\item for any $n\ge 1$ and a sequence of objects $x_0,\dots,x_n\in\Ob(C)$ of degree $1-n$
	$$
	F_n\colon C(x_{n-1},x_n)\otimes\dots \otimes C(x_1,x_2)\otimes C(x_0,x_1)\to D(F_0(x_0),F_0(x_n))[1-n]
	$$
	such that one has:
	\begin{equation}\label{ainftysigns}
	\begin{aligned}
	\ &d(F_n(f_n\otimes\dots\otimes f_1))+\\
	&\sum_{a+b=n}(-1)^{b-1+(a-1)(|f_1|+\dots+|f_b|)}F_a(f_n\otimes\dots\otimes f_{b+1})\cdot F_b(f_b\otimes \dots \otimes f_1)=\\
	&\sum_{k=0}^{n-1}(-1)^{n-1+|f_1|+\dots+|f_k|}F_n(f_n\otimes\dots\otimes f_{k+2}\otimes d(f_{k+1})\otimes f_k\otimes\dots\otimes f_1)+\\
	&\sum_{k=0}^{n-2}(-1)^kF_{n-1}(f_n\otimes\dots\otimes f_{k+3}\otimes (f_{k+2}\circ f_{k+1})\otimes f_k\otimes \dots\otimes f_1)
	\end{aligned}
	\end{equation}
\end{itemize}	

}
\end{defn}

\begin{defn}\label{defintro1}{\rm
		Let $C,D$ be (unital) dg categories over $\k$, $F\colon C\to D$ an $A_\infty$ functor, $\{F_i\}_{i\ge 1}$ its Taylor components.
		\begin{itemize}
			\item[(1)] $F$ is called {\it strongly unital} if $F_1(\id_X)=\id_{F(X)}$ for any object $X\in C$, and $F_n(\dots, \id_X, \dots)=0$ for any object $X$ and $n\ge 2$,
			\item[(2)] $F$ is called {\it weakly unital} if $F_1(\id_X)=\id_{F(X)}$ for any object $X\in C$ (and the second condition is dropped).
		\end{itemize}
	}		
\end{defn}		

\section{\sc A cofibrant resolution of a unital dg algebra}
Here we provide a canonical unital cofibrant dg algebra, quasi-isomorphic to a unital dg algebra $A$ over a field $\k$. (We consider only the case of a dg algebra for simplicity, the construction is directly generalised for the case of a small dg cateory).

The classical bar-cobar resolution of $A$ fails to be unital, it is only weakly unital dg algebra. A well-known explicit unital construction comes from the {\it curved} version of bar-cobar duality, due to L.Positselski [Pos] (see also [Ly3]). A drawback of this construction is that it is not canonical. 

We haven't seen this construction in the literature. In our opinion, it deserves to be included as an appendix to this paper, by two reasons. At first, it ``replaces'' the bar-cobar resolution, which fails to be strictly unital, which was one of our starting points. At second, it can be easily generalised to a cofibrant resolution of $i(C)$ in $\Cat_\dgwu(\k)$, for $C\in \Cat_\dg(\k)$. 

Let $A$ be a unital dg algebra over $\k$. Consider the dg algebra $A_+=A\oplus \k[1]$. It is a unital dg algebra with unit $1_A$, and the product of $A$ with $\k[1]$, as well as of $\k[1]$ with itself, is defined as 0. 

Consider the bar-complex 
$$
\Bar(A_+)=\bigoplus_{n\ge 1}A_+[1]^{\otimes n}
$$
which is a dg coalgebra.
We use the notation $\xi$ for a generator of $\k[1]$. Then  a general monomial element of $\Bar(A_+)$ is denoted as $$
a_1\otimes\dots\otimes a_{i_1}\otimes \xi\otimes a_{i_1+1}\otimes\dots\otimes a_{i_2}\otimes \xi\otimes a_{i_2+1}\otimes\dots
$$
Now consider the {\it unital} dg algebra $$C_0(A)=\Cobar_+(\Bar(A_+))$$
where, for a dg coalgebra $B$,
$$
\Cobar_+(B)=\k\oplus \bigoplus_{n\ge 1}B[-1]^{\otimes n}
$$
with the cobar-differential. It is a unital dg algebra. Denote by $1_\k$ the unit of $\k$. It is the unit of $C_0(A)$.
We denote the product in $\Cobar_+(B)$ by $\boxtimes$.

Consider a derivation $d_\xi$ of degree +1 of $C_0(A)$ whose only non-zero Taylor coefficient is linear, and is defined as 
\begin{equation}
\begin{aligned}
\ &d_\xi|_A=0,\ d_\xi(\xi)=1_\k-1_A\\
&d_\xi(x_1\otimes \dots \otimes x_n)=\sum_{\ell=1}^k
\pm x_1\otimes \dots \otimes x_{i_\ell-1}\otimes 1_A\otimes x_{i_\ell+1}\otimes\dots\otimes x_n\\
&\text{where $x_{i_1}=\dots=x_{i_k}=\xi$ and other $x_i\in A$}
\end{aligned}
\end{equation}
One has
\begin{lemma}
	The differential $d_\xi$ squares to 0, and $d_\xi$ commutes with $d_\Bar+d_\Cobar$. Consequently, $d_\tot:=d_\Bar+d_\Cobar+d_\xi$ squares to 0.
\end{lemma}

It is a direct computation.

\qed

We denote 
$$
C(A)=(C_0(A), d_\Bar+d_\Cobar+d_\xi)
$$
It is a unital dg algebra.

\begin{prop}
There is a unital dg algebra map $p\colon C(A)\to A$ which is a quasi-isomorphism.
\end{prop}
\begin{proof}
We start with computing the cohomology of $(C(A), d_\tot)$.
The differential $d_\Bar+d_\Cobar$ preserves the total number of $\xi$-factors in a (homogeneous in $\xi$) element of $C(A)$. It makes $C(A)$ a bicomplex. Define 
$$
\deg_\Bar(x_1\otimes\dots \otimes x_n)=-n+\sum_i \deg_0 x_i
$$
where $\deg_0(a)=\deg_A(a)$ and $\deg_0(\xi)=0$.
Next, define 
$$
\deg_1(\omega_1\boxtimes\dots\boxtimes \omega_k)=k+\sum_i \deg_\Bar(\omega_i)
$$
$$
\deg_\Cobar(\omega_1\boxtimes\dots\boxtimes \omega_k)=k
$$
$$
\deg_\Bar(\omega_1\boxtimes\dots\boxtimes \omega_k)=\sum_i \deg_\Bar(\omega_i)
$$
where $\omega_i\in \Bar(A_+)$.
Finally, define $$
\deg_\xi(\alpha)=-(\sharp(\xi)\text{\ in $\alpha$}), \ \alpha\in C(A)
$$

$$
\deg_\tot(\alpha)=\deg_1(\alpha)+\deg_\xi(\alpha)
$$
where $\deg_\tot$ is the degree of $\alpha$ in $C(A)$. 

Thus $C(A)$ becomes a bicomplex, with $C(A)^{a,b}$ defined as the spaces of elements $\alpha\in C(A)$ with $\deg_1\alpha=a$, $\deg_\xi(\alpha)=b$. 

We compute the cohomology of $C(A)$ by using a spectral sequence, which computes the cohomology of $d_\Bar+d_\Cobar$ at first. The bicomplex lives in the $II$ and $III$ quarters, so the spectral sequence converges.

The term $E_1^{a,b}$ of the spectral sequence is equal to 
$H^a(C(A)^{\udot,b}, d_\Bar+d_\Cobar)$. Thus, we have to compute the cohomology of the complex $(C(A)^{\ldot,b}, d_\Bar+d_\Cobar)$. Denote this complex by $C^\udot_b$.

The complex $C^\udot_b$ (for $b$ fixed) is by its own a bicomplex, 
with differentials $d_\Bar$ and $d_\Cobar$. Thus $C_b^{m,n}$ consist of all elements $\alpha\in C_b^\udot$ with $\deg_\Bar(\alpha)=m$ and $\deg_\Cobar(\alpha)=n$ (in this case, $\deg_\tot(\alpha)=m+n+b$). 

The spectral sequence, whose first differential is $d_\Cobar$, converges (the other possible spectral sequence, whose first differential id $d_\Bar$, generally diverges). 

We denote by $\mathbb{E}(b)_k^{m,n}$ the $k$-th term of this spectral sequence. We have
$$
\mathbb{E}(b)_1^{m,n}=H^m(\mathbb{E}(b)_1^{\udot,n}, d_\Cobar)
$$

\begin{lemma}
One has: 
\begin{equation}
\mathbb{E}(b)_1^{m,n}=\begin{cases}
0,& m\ne 0,1 \text{ or } m=0, n\ne 0\\
\k,& b=0, m=0, n=0\\
\k[1], &b=-1, m=1, n=-1\\
A^n,& b=0,m=1\\
0,& b\ne 0,-1
\end{cases}
\end{equation}
where $(-)^n$ stands for degree $n$ elements. 
\end{lemma}
\begin{proof}
	The argument is standard and comes from the following observation. Let $V$ be a (graded) vector space, consider the cofree non-unital coalgebra $T^\vee_{\ge 1}(V)=\oplus_{n\ge 1}V^{\otimes n }$. Then $\Cobar(T^\vee_{\ge 1}(V))$ is quasi-isomorphic to $V[-1]$. This statement is proven using Koszul duality. 
\end{proof}

It follows that the spectral sequence $\mathrm{E}_\ldot$ collapses at the $\mathbb{E}_1$ term. 

Now turn back to the spectral sequence $E_\ldot$. 

\begin{lemma}
One has
\begin{equation}
 E_1^{a,b}=\begin{cases}
 \k[1],&a=0, b=-1\\
 (\k\oplus A)^a, &b=0\\
 0,&\text{otherwise}
 \end{cases}
 \end{equation}
\end{lemma}
The differential $d_1$ is induced by $d_\xi$. It looks like
$$
{\k}[1]\xrightarrow{d_1} \k\oplus A, \ \ d_1\colon 1\mapsto 1_\k-1_A
$$
Its cohomology is isomorphic to $A$. The spectral sequence $E_\ldot$ collapses at the $E_2$ term.

It completes the computation of cohomology of $C(A)$. 

Now define a map of dg algebras $p\colon C(A)\to A$ on generators
\begin{equation}
\begin{aligned}
&p|_{(\k[1]\oplus A)^{\otimes n}}=0,\ n\ge 2\\
&p(\xi)=0\\
&p|_A=\id\\
&p(1_\k)=1_A
\end{aligned}
\end{equation}
and extend it to $C(A)$ as a map of algebras. 

It follows from the previous computation that $p$ is a quasi-isomorphism. 

\end{proof}

\section{\sc Push-outs in the category $\Cat_\dgwu(\k)$}\label{appc}
	Consider the following push-out diagram:
		\begin{equation}\label{pushoutc}
		\begin{tikzpicture}[baseline= (a).base]
		\node[scale=1] (a) at (5,5){
			\begin{tikzcd}
			\C \arrow[d, "i_2"'] \arrow[r, "i_1"] & \C_1 \\
			\C_2           & 
			\end{tikzcd}
		};
		\end{tikzpicture}
		\end{equation}
		where $i_a: \C \hookrightarrow \C_a, a=1,2$ are embeddings of wu dg categories. We provide an explicit formula for the push-out $\E$. This formula is essentially used in the proof of Theorem \ref{bigtheorem}.\\
	Note that this colimit $\E$ is equivalently the colimit of the following coequalizer diagram: 
		\begin{equation}\label{c1c2}
		\begin{tikzpicture}[baseline= (a).base]
		\node[scale=1] (a) at (5,5){
			\begin{tikzcd}
			\C \arrow[r, shift right] \arrow[r, shift left] & \C_1 \oplus \C_2 
			\end{tikzcd}
		};
		\end{tikzpicture}
		\end{equation}
		where the maps are $(i_1,0)$ and $(0,i_2)$.
	The category structure on  $\C_1 \oplus \C_2$ is defined as
		\[Ob(\C_1 \oplus \C_2) := Ob(\C_1) \sqcup Ob(\C_2) \]
		\begin{equation*}
		Hom_{\C_1 \oplus \C_2}(x,y) =
		\begin{cases}  Hom_{\C_a}(x,y) &  \text{ if } x, y \in Ob(\C_a) \\ 
		0 &  \text{ otherwise. } \end{cases} 
		\end{equation*}
		\comment
		In particular $p_n(...)$ is defined only on tuples of morphisms all belonging to the same wudg category $\C_j$. \\
		We can give an explicit description of $\E$ following the procedure described in the proof of the existence of the coequalizers [\cite{PS}, Proposition 1.18], namely exploiting the following diagram:
		\begin{equation}
		\begin{tikzpicture}[baseline= (a).base]
		\node[scale=1] (a) at (5,5){
			\begin{tikzcd}
			FUFU(\C) \arrow[d, shift right] \arrow[d, shift left] \arrow[r, shift right] \arrow[r, shift left] & FUFU(\C_1)\oplus FUFU(\C_2) \arrow[d, shift right] \arrow[d, shift left] \arrow[r] & FE' \arrow[d, shift right, dotted] \arrow[d, shift left, dotted] \\
			FU(\C) \arrow[d] \arrow[r, shift right] \arrow[r, shift left] \arrow[u, bend left] & FU(\C_1)\oplus FU(\C_2) \arrow[r] \arrow[d] \arrow[u, bend left] & FE \arrow[d, dotted] \arrow[u, bend left, dotted]  \\
			\C \arrow[r, shift right] \arrow[r, shift left]          & \C_1\oplus\C_2 \arrow[r,dotted]           & \E                    
			\end{tikzcd}
		};
		\end{tikzpicture}
		\end{equation}
		where we get $\E$ as the coequalizer of \[ \begin{tikzcd} FE' \arrow[r, shift right] \arrow[r, shift left] & FE \end{tikzcd}\]
		The matter is that new ``free" $p_n(...)'s$ and $m(..)$ are created in both $FE'$ and $FE$; indeed, in both the coequalizers $E$ and $E'$ we glue the objects of the dg-graphs $U\C_1$ and $U\C_2$ (or of $UFU\C_1$ and $UFU\C_2$) which are image of the same object of $\C$, and then we apply the free weakly unital dg-category functor $F$, therefore creating new $p_n(...)'s$ and $m(..)$ even between ``mixed" morphisms. \\
		\endcomment
		In [PS, Prop.1.18] we considered the general coequalizers in the category $\Cat_\dgwu(\k)$. Here we provide the corresponding description for the case of the coequalizer \eqref{c1c2}. Here we essentially use that $i_1,i_2$ are fully-faithful functors. The derivation of this description from loc.cit. is straightforward. 
		
		We use notation $\bar{a}$ which is defined as $\bar{1}=2$, $\bar{2}=1$.
		\begin{prop}\label{propappc}
		Assume $i_1,i_2$ in \eqref{pushoutc} are fully faithful. Then the push-out weakly unital dg category $\E$ has the following description.\\
		The objects of $\E$ are given by the coequalizer of sets:
		\[ Ob(\E) = Ob(\C_1) \sqcup Ob(\C_2) / \sim \]
		where $\sim$ is the equivalence relation generated by: $i_1(x) \sim i_2(x), \forall x \in Ob(\C)$. \\
		Consider $x,y \in Ob(\C_a)$, $a=1,2$.  Then:
	\begin{equation}\label{pushoutcase1}
		\begin{aligned}
		\E(x,y) = &\C_a(x,y)\bigoplus \bigoplus_{u,v\in\C} \mathcal{O}^\prime(3)\otimes \C_a(v,y)\otimes\Cbar_{\abar}(u,v)\otimes\C_a(x,u)  \bigoplus\\
	&\bigoplus_{u,v,u_1,v_1\in\C}\mathcal{O}^\prime(5)\otimes \C_a(v_1,y)\otimes \C_\abar(u_1,v_1)\otimes\C_a(u,v_1)\otimes \C_\abar(v,u)\otimes \C_a(x,v)\bigoplus\dots
	\end{aligned}
\end{equation}
	where we identify an object $u\in \C$ with its images $i_a(u)\in \C_a$, and $\Cbar_a(u,v) := \C_a(u,v) / i_a(\C(x,y))$. \\	
	For $x \in Ob(\C_1)\setminus \Ob(\C)$, $y \in Ob(\C_2)\setminus\Ob(\C)$, one has:
\begin{equation}\label{pushoutcase2}
\begin{aligned}
\E(x,y) :=\bigoplus_{u\in\C}\mathcal{O}^\prime(2)\otimes \C_2(u,y)\otimes \C_1(x,u) \bigoplus \bigoplus_{u,v,w\in\C}\mathcal{O}^\prime(4)\otimes \C_2(w,y)\otimes {\Cbar_1(v,w)} \otimes {\Cbar_2(u,v)}\otimes \C_1(x,u) \bigoplus \dots
\end{aligned}
\end{equation}
\end{prop}

\bigskip
{\noindent\sc P.P.:}

{\small
	\noindent {\sc Universiteit Antwerpen, Campus Middelheim, Wiskunde en Informatica, Gebouw G\\
		Middelheimlaan 1, 2020 Antwerpen, Belgi\"{e}}}

\bigskip

\noindent{{\it e-mail}: {\tt Piergiorgio.Panero@uantwerpen.be}}

\bigskip
{\noindent\sc B.Sh.:}

{\small
	\noindent {\sc Universiteit Antwerpen, Campus Middelheim, Wiskunde en Informatica, Gebouw G\\
		Middelheimlaan 1, 2020 Antwerpen, Belgi\"{e}}}
\bigskip

{\small
	\noindent{\sc Laboratory of Algebraic Geometry,
		National Research University Higher School of Economics,
		Moscow, Russia}}

\bigskip

\noindent{{\it e-mail}: {\tt Boris.Shoikhet@uantwerpen.be}}

\end{document}